\documentclass[twoside,11pt]{article} 
\usepackage{jmlr2e}
\usepackage{natbib}

\usepackage[applemac]{inputenc}

\usepackage{amsmath,amssymb} 
\usepackage{url,graphicx,algorithm2e, hyperref, nameref,xkeyval}
\usepackage{graphicx,epsfig,color}
\usepackage{amsmath,mathrsfs} 
\usepackage{amssymb} 
\usepackage{amsfonts}
\usepackage{mathtools}
\usepackage{epstopdf}
\usepackage{ifthen}
\usepackage{bbm}
\graphicspath{{eps/}}
\usepackage{bookmark}

\usepackage{comment}
\includecomment{extra}
\excludecomment{k}
\includecomment{krad}

\newcommand{\inR}{\in \mathbb{R}}

\newcommand{\C}{ \mathbb{C}}
\newcommand{\R}{ \mathbb{R}}

\newcommand{\N}{ \mathbb{N}}

\newcommand{\eqdef}{\stackrel{\vartriangle}{=}}
\newcommand{\Top}{\mathsf{T}}
\newcommand{\Lop}{{\rm L}}

\newcommand{\dint}{{\rm d}}

\newcommand{\Fourier}{ \mathcal{F}} 

\newcommand{\bx}{{\boldsymbol x}}
\newcommand{\bw}{{\boldsymbol \omega}}

\newcommand{\bk}{{\boldsymbol k}}

\def\V#1{{\boldsymbol{#1}}}         
\def\Spc#1{{\mathcal{#1}}}  
\def\M#1{{\bf{#1}}}  
\def\Op#1{{\mathrm{#1}}}  
\def\ee{\mathrm{e}} 
\def\jj{\mathrm{j}}

\def\Proj{\mathrm{Proj}} 
\def\Identity{\mathrm{Id}} %

\newcommand{\embedC}{\xhookrightarrow{}}
\newcommand{\embedD}{\xhookrightarrow{\mbox{\tiny \rm d.}}}
\newcommand{\embedIso}{\xhookrightarrow{\mbox{\tiny \rm iso.}}}

\renewcommand{\[}{\begin{equation}}
\renewcommand{\]}[1]{\label{eq:#1}\end{equation}}

\title{Ridges, Neural Networks, and the Radon Transform
}

\editor{XXX}
\author{\name {Michael Unser} \email {michael.unser@epfl.ch}\\
\addr Biomedical Imaging Group\\ \'Ecole polytechnique f\'ed\'erale de Lausanne (EPFL)\\ CH-1015 Lausanne, Switzerland
}

\begin{document}

\maketitle




%
\begin{abstract}
A ridge is a function that is characterized by a one-dimensional profile (activation)
and a multidimensional direction vector. Ridges appear in the theory of neural networks as functional descriptors of the effect of a neuron, with the direction vector being encoded in the linear weights. In this paper, we investigate 
properties of the Radon transform in relation to ridges and to the characterization of neural networks. We introduce a broad category of hyper-spherical Banach subspaces (including the relevant subspace of measures) over which the back-projection operator is 
invertible. 
We also give conditions under which the back-projection operator is extendable to the full parent space with its null space being 
identifiable as a Banach complement. 
Starting from first principles, we then characterize the sampling functionals that are in the range of the filtered Radon transform. 
Next, we extend the definition of ridges for any distributional profile 
and determine their (filtered) Radon transform in full generality.
Finally, we apply our formalism to clarify and simplify some of the results and proofs on the optimality of ReLU networks that have appeared in the literature.
\end{abstract}

\tableofcontents

\section{Introduction}
A ridge is a multidimensional function $\V x \mapsto r(\V w^\Top \V x)$ from $\R^d \to \R$ that is characterized by a 1D profile $r: \R \to \R$ and a weight vector $\V w \in \R^d \backslash\{\V 0\}$ \citep{Pinkus2015}. Ridges are ubiquitous in mathematics and engineering. Most significantly, the elementary unit (neuron) in a neural network is a function of the form $f_k(\V x)=\sigma(\V w_k^\Top \V x -t_k)$, which is a ridge with a shifted profile $r=\sigma(\cdot-t_k)$, where $\sigma: \R \to \R$ is the activation function and where $t_k\in \R$ (bias) and $\V w_k \in \R^d$ (linear weights) are the trainable parameters of the $k$th neuron \citep{Bishop2006}. Variants of the universal-approximation theorem ensure that any continuous function can be approximated as closely as desired by a weighted sum of ridges with a fixed activation under 
mild conditions on $\sigma$ \citep{Cybenko1989,Hornik1989,Barron1993}.

Ridges are also intimately tied to the Radon transform \citep{Logan1975,Madych1990d} under the condition that the weight vector $\V w$ has a unit norm, so that $\V w \in  \mathbb{S}^{d-1}$ where $\mathbb{S}^{d-1}$ is the unit sphere in $\R^d$ whose generic elements will be denoted by $\V \xi$. This connection is exploited in the ridgelet transform, which provides a wavelet-like representation of functions where the basis elements are ridges \citep{Murata1996,Rubin1998,Candes1999_ridgelets,Candes1999,Kostadinova2014}. 
The expansion of a function in terms of ridgelets is a precusor to sparse signal approximation. There, the idea is to represent a function by a linear combination of a small number of atoms taken within a dictionary \citep{Elad2010b,Foucart2013}. This paradigm, which is the basis for compressed sensing \citep{Donoho2006, Candes2007}, has been adapted to shallow neural networks 
by considering a dictionary that consists of a continuum of neurons. Mathematically, this can be implemented through the integral representation 
(infinite-width neural network) 
\begin{align}
\label{Eq:IntegralNeural}
f(\V x)&=\int_{\R \times \mathbb{S}^{d-1}}\sigma(\V \xi^\Top \V x -t) \dint \mu(t,\V \xi),
\end{align} where $\mu$ is a measure
on $\R \times \mathbb{S}^{d-1}$ (hyper-spherical domain). This model is fitted to data subject to a penalty on the total-variation norm of $\mu$ \citep{Bach2017}. Remarkably, this infinite-dimensional convex optimization problem results in sparse minimizers of the form $\mu=\sum_{k=1}^K a_k\delta_{\V z_k}$ with $\V z_k=(t_k,\V \xi_k) \in \R \times \mathbb{S}^{d-1}$, which then map into standard two-layer neural networks \citep{Bach2017}. 
Interestingly, we can relate
 \eqref{Eq:IntegralNeural}  to the Radon transform by identifying $\mu$ as the (generalized) function $g_\mu$ (with $\dint \mu(t,\V \xi)=g_\mu(t,\V \xi)\dint t\dint \V \xi$) and by rewriting the integral as
\begin{align}
f(\V x)&=\int_{ \mathbb{S}^{d-1}} \left( \int_{\R} \sigma(\V \xi^\Top \V x -t) g_\mu(t,\V \xi) \dint t \right)\dint \V \xi
= \Op R^\ast\Op L_{\rm rad}\{ g_\mu\}(\V x),
\label{Eq:IntegralNeural3}
\end{align}
where $\Op R^\ast$ (the adjoint of the Radon transform) is the  back-projection operator of computer tomography \citep{Natterer1984}. Our ``radial'' operator $\Op L_{\rm rad}: g_\mu \mapsto \sigma \ast_{t} g_\mu$ on the right-hand side of \eqref{Eq:IntegralNeural3}  implements the Radon-domain convolution with $\sigma$ 
along the variable $t$.

While the synthesis approach to the learning problem proposed by \citet{Bach2017} is insightful, there is a strong incentive to make the connection with regularization theory in direct analogy with the classical theory of learning that relies on reproducing-kernel Hilbert spaces \citep{Wahba1990,Poggio1990,Scholkopf1997,Scholkopf2001,Alvarez2012,Unser_2020}. This is feasible provided that the linear relation between $f$ and $g_\mu$ expressed by \eqref{Eq:IntegralNeural3} be one-to-one. This requires that the operators $\Op L_{\rm rad}$ and $\Op R^\ast$ in \eqref{Eq:IntegralNeural3}
be both  invertible. 
\citet{Ongie2020b} made an important step in that direction by  showing that ReLU networks are minimizers of a Radon-domain total-variation norm that involves the 
 Laplacian of $f$. Their optimality result was then generalized by \citet{Parhi2021b}
who considered a broader class of differential operators inspired by spline theory \citep{Unser2017}. The leading idea there is that the operator $\Op L_{\rm rad}$ in \eqref{Eq:IntegralNeural3} should implement some variant of an $n$th-order integrator, with $\sigma={\rm ReLU}$ being the solution for $n=2$. Such an $\Op L_{\rm rad}$ can be formally inverted by applying an $n$th-order partial derivative (e.g., $\Op L^{-1}_{\rm rad}=\partial^n_t$), which motivates the use of the latter as (filtered) Radon-domain regularization operator. 

The proposed spline-based approach to the inversion of \eqref{Eq:IntegralNeural3} is elegant and intuitively appealing.
However, the formulation and resolution of the corresponding optimization problem requires special care because the underlying function spaces have a nontrivial kernel (null space) that needs to be factored out. The latter statement applies not only to the regularization operator (e.g., Laplacian and/or Radon-domain radial derivatives) but also to $\Op R^\ast$, which is an aspect that has been overlooked.
While there is a rich theory on the invertibility of the Radon transform \citep{Helgason2011,Rubin1998,Boman2009,Ramm2020}, there are comparatively fewer---and not as strong---results on the invertibility of $\Op R^\ast$, the problem being that this operator has a huge null space \citep{Ludwig1966}. The primary spaces on which $\Op R^\ast$ is known to be injective, and hence invertible, are
\begin{itemize}
\item $\Spc S(\mathbb{P}^d)$ (the even part of Schwartz' hyper-spherical---or Radon-domain---test functions) \citep{Solmon1987};
\item $L_{\infty, c}(\mathbb{P}^d)$ (the bounded even functions of compact support) \citep{Ramm1996};  
\item $\Spc S'_{\rm Liz}(\mathbb{P}^d)$ (the even Lizorkin distributions) \citep{Kostadinova2014}.
\end{itemize}
The space $\Spc S'_{\rm Liz}(\R^d)$ of Lizorkin distributions, which is the topological dual of $\Spc S_{\rm Liz}(\R^d)$ (the subspace of Schwartz functions that are orthogonal to all polynomials), is especially attractive in that context. Indeed, the Radon transform being an homeomorphism from $\Spc S'_{\rm Liz}(\R^d)$
onto $\Spc S'_{\rm Liz}(\mathbb{P}^d)$, the inversion process is straightforward \citep{Kostadinova2014}. Lizorkin distributions also interact very nicely with the Laplace operator, which makes then well suited to the investigation of fractional integrals \citep{Samko1993} and of wavelets \citep{Saneva_2010}.
The Lizorkin framework, however, has one basic limitation. The underlying objects---Lirzorkin distributions---are abstract entities, with $\Spc S'_{\rm Liz}(\R^d)$ being isomorphic to the quotient space $\Spc S'(\R^d)/\Spc P$, where $\Spc S'(\R^d)$ and $\Spc P$ are the spaces of tempered distributions and polynomials, respectively. Thus, Lizorkin distributions are generally identifiable only modulo some polynomial.
Fortunately, this is not a problem when dealing with ordinary functions $f \in L_p(\R^d)$ since $L_p(\R^d)$ is continuously embedded in $\Spc S'_{\rm Liz}(\R^d)$ for any $p>1$ \citep{Samko1982denseness}. This implies that the Lizorkin distribution $f + \Spc P \in \Spc S'_{\rm Liz}(\R^d)$ has a unique ``concrete'' representer $f \in L_p(\R^d)$, which amounts to simply setting the polynomial to zero.
The situation, however, is not as clearcut for functions and ridge profiles that exhibit polynomial growth at infinity. To offer  insights on the nature of the problem, let us consider three distinct neuronal units $f_1(x)=(x-t_k)_+$ (ReLU activation), $f_2(x)=\tfrac{1}{2}|x-t_k|$, and $f_3(x)=x+(x-t_k)_+$ (ReLU with skip connection), which are all valid representers of the same Lizorkin distribution $f=f_i+\Spc P \in \Spc S'_{\rm Liz}(\R)$ for $i=1,2,3$ (since the $f_i$'s only differ by a first-order polynomial). Suppose that a theoretical argument can be made concerning the optimality of the announced $f \in \Spc S'_{\rm Liz}(\R)$. The practical difficulty then is to map this result into a concrete architecture. Should the choice be one of the $f_i$, if any? The least we can say is that the convenient rule of ``setting the polynomial to zero'' is not applicable here because it is unclear what the underlying polynomial truly is. This intrinsic ambiguity jeopardizes some of the conclusions regarding the connection between ReLU neural networks, ridge splines, and the Radon transform that have been reported in the literature \citep{Sonoda2017,Parhi2021b}. We are in the opinion that adjustments are needed.

In this paper, we revisit the topic and extend the existing formulation so that it can handle arbitrary ridge profiles, without (polynomial) ambiguity. Our four primary contributions are as follows.

\begin{itemize}
\item A detailed investigation of the invertibility of the back-projection operator for a broad family of Radon-compatible Banach subspaces $\Spc X'_{\rm Rad}\subset \Spc X'$, where $\Spc X'$ is the topological dual of some generic hyper-spherical parent space $\Spc X$. 

\item A constructive characterization of the extreme points of 
the space of Radon-compatible hyper-spherical measures.

\item The extension of ridges to distributional profiles $r \in \Spc S'(\R)$ 
and the determination of their Radon transform.

\item The application of the formalism to the investigation of a functional-optimization problem that results in solutions that are parameterized by ReLU neural networks. Our contribution there is to clarify the analysis of \citet{Parhi2021b} and to provide a characterization of the full solution set.

\end{itemize}

The paper is organized as follows: We start with notations and mathematical preliminaries in Section 2. In particular, we recall the main properties of the classical Radon transform and its adjoint, and show how to extend them to tempered distributions by duality.
In Section \ref{Sec:back-projection}, we develop a formulation that leads to the identification of a generic family of Radon-compatible Banach spaces over which the back-projection operator is invertible. Our results are summarized in Theorem \ref{Theo:Complementedspaces}, which can be viewed as the Banach counterpart of the classical result for tempered distributions \citep{Ludwig1966}.
In Section \ref{Sec:RadonDirac}, we use our 
framework to characterize the sampling functionals (Radon-compatible Diracs) that are in the range of the filtered Radon transform 
(Theorem \ref{Theo:Extreme0}). 
In Section \ref{Sec:Ridges}, we introduce a general definition of a ridge with an arbitrary distributional profile 
and derive its
(filtered) Radon transform. Finally, in Section \ref{Sec:RadonRegul}, we apply our formalism to the resolution of a multidimensional supervised-learning problem with a 2nd-order Radon-domain regularization formulated by \citet{Parhi2021b}, the outcome being Theorem \ref{Theo:ReLUSplines} on the optimality of ReLU networks.

\section{Mathematical Preliminaries}
\subsection{Notations}
\label{Sec:Notations}
We shall consider multidimensional functions $f$ on $\R^d$ that are indexed by the variable $\V x \in \R^d$. To describe their partial derivatives, we use the multi-index 
$\V k=(k_1,\dots,k_d) \in \N^d$ with the notational conventions $\V k!\eqdef\prod_{i=1}^d k_i!$, $|\V k|\eqdef k_1+\cdots+k_d$,
$\V x^{\V k}\eqdef\prod_{i=1}^d x_i^{k_i}$ for any $\V x \in \R^d$, and $\partial^\V k f(\V x)\eqdef\frac{\partial^{|\V k|}f(x_1,\dots,x_d)}{\partial^{k_1}_{x_1} \cdots \partial^{k_d}_{x_d}}$. This allows us to write the multidimensional Taylor expansion around $\V x_0$ of an analytical function $f: \R^d \to \R$ explicitly as
\begin{align}
\label{Eq:Taylor}
f(\V x)=\sum_{n=0}^\infty \sum_{|\V k|=n} \frac{\partial^\V k f(\V x_0) (\V x- \V x_0)^\V k}{\V k! },
\end{align}
where the internal summation is over all multi-indices $\V k=(k_1,\dots,k_d)$ such that $k_1+\cdots+k_d=n$.

Schwartz' space of smooth and rapidly decreasing test functions $\varphi: \R^d \to \R$ equipped with the usual Fr\'echet-Schwartz topology is denoted by $\Spc S(\R^d)$. Its continuous dual is the space $\Spc S'(\R^d)$ of tempered distributions. In this setting, the Lebesgue spaces $L_p(\R^d)$ for $p\in[1,\infty)$ can be specified as the completion of 
$\Spc S(\R^d)$ equipped with the $L_p$-norm $\|\cdot\|_{L_p}$; that is, $L_p(\R^d)=\overline{(\Spc S(\R^d),\|\cdot\|_{L_p})}$.
For the end point $p=\infty$, we have that $\overline{(\Spc S(\R^d),\|\cdot\|_{L_\infty})}=C_0(\R^d)$ with $\|\varphi\|_{L_\infty}=\sup_{\V x \in \R^d}|\varphi(\V x)|$, where $C_0(\R^d)$ is the space of continuous functions that vanish at infinity. Its continuous dual is the space $\Spc M(\R^d)=\{f \in \Spc S'(\R^d): \|f\|_{\Spc M}<\infty\}$ of bounded 
Radon measures with
$$
\|f\|_{\Spc M}=\sup_{\varphi \in \Spc S(\R^d): \|\varphi\|_{L_\infty}\le 1} \langle f, \varphi \rangle.
$$
The latter is a superset of $L_1(\R^d)$, which is isometrically embedded in it, meaning that  $\|f\|_{L_1}=
\|f\|_{\Spc M}$ for all $f \in L_1(\R^d)$.

The 
Fourier transform of a function $\varphi \in L_1(\R^d)$ is defined as
\begin{align}
\widehat \varphi(\bw)\eqdef \Fourier\{\varphi\}(\bw)=\frac{1}{(2 \pi)^d} \int_{\R^d} \varphi(\V x) \ee^{-\jj \langle \bw, \V x \rangle} \dint \V x.\label{Eq:Fourier}
\end{align}
Since the Fourier operator $\Fourier$ continuously maps $\Spc S(\R^d)$ into itself, the transform can be extended by duality to the whole space $\Spc S'(\R^d)$ of tempered distribution. Specifically, $\widehat f=\Fourier\{f\} \in \Spc S'(\R^d)$ is the (unique) {\em generalized Fourier transform} of $f\in \Spc S'(\R^d)$ if and only if $\langle \widehat f, \varphi \rangle=\langle f, \widehat\varphi \rangle$ for all $\varphi  \in \Spc S(\R^d)$, where $\widehat\varphi=\Fourier\{\varphi\}$  is the ``classical'' Fourier transform of $\varphi$ defined by \eqref{Eq:Fourier}.

\subsection{Polynomial Spaces and Related Projectors}
\label{Sec:Null}
The regularization operators (e.g., the Laplacian) that are of interest to us are isotropic and have a growth-restricted null space formed by the space $\Pi_{n_0}$ of polynomials of degree $n_0$.
Here, we choose to expand these polynomials 
in the monomial/Taylor basis
\begin{align}
\label{Eq:TaylorBasis}
m_\V k(\V x)\eqdef\frac{\V x^{\V k}}{\V k!}
\end{align}
with $|\bk|\le n_0$. Defining the  $\ell_2$-norm of the Taylor coefficients of the polynomial
$p_0=\sum_{|\V k|\le n_0} b_\V k m_{\V k}$ as
\begin{align}
\|p_0\|_{\Spc P}\eqdef \|(b_\V k)_{|\V k|\le n_0}\|_2,
\end{align}
we add a topological structure that results in the description
\begin{align}
\label{Eq:PolNullspace}
\Spc P_{n_0}=\{p_0\in \Pi_{n_0}: \|p_0\|_{\Spc P}<\infty\}.
\end{align}
The important point here is that  \eqref{Eq:PolNullspace} specifies a finite-dimensional Banach subspace of $\Spc S'(\R^d)$. Its continuous dual $\Spc P'_{n_0}$ is finite-dimensional as well, although it is composed of ``abstract'' elements 
that are, in fact, equivalence classes in $\Spc S'(\R^d)$. Yet, it possible to
 identify every dual element $p_0^\ast\in \Spc P'_{n_0}$ concretely as a function by selecting a particular dual basis $\{m_{\V k}^\ast\}_{|\V k|\le n_0}$ such that $\langle m^\ast_{\V k}, m_{\V k'}\rangle=\delta_{\V k-\V k'}$ (Kroneker delta). Our specific choice is 
 \begin{align}
 \label{Eq:Dualbasis}
m^\ast_\V k=(-1)^{|\V k|}\partial^{\V k}\kappa_{\rm iso} \in \Spc S(\R^d)
 \end{align}
 with $\V k\in \N^d$,  where $\kappa_{\rm iso}$ is the isotropic function described below. 
 
\begin{lemma}
\label{Theo:IsotropicWindow}
There exists an entire isotropic function $\kappa_{\rm iso} \in \Spc S(\R^d)$ with $0\le \widehat{\kappa}_{\rm iso}(\bw)\le 1$ and $\widehat{\kappa}_{\rm iso}(\bw)=0$ for $\|\bw\|\ge 1$
such that
\begin{align}
\int_{\R^d} \frac{\V x^{\V k}}{\V k!} (-1)^{|\V n|}\partial^{\V n}\kappa_{\rm iso}(\V x)\dint \V x=
\; \delta_{\V k -\V n} \end{align} for all $\V k, \V n \in \N^d$.
\end{lemma}
\begin{proof}
We take $\kappa_{\rm iso}=\Fourier^{-1}\{\widehat \kappa_{\rm rad}(\|\cdot\|)\}$, where the radial profile $\widehat \kappa_{\rm rad}: \R \to \R$  is such that $\widehat \kappa_{\rm rad} \in \Spc S(\R)$,
$\widehat \kappa_{\rm rad}(\omega)=1$ for $0\le |\omega| \le R_0\le \tfrac{1}{2}$, and $\widehat \kappa_{\rm rad}(\omega)=0$ for $|\omega|\ge1$. A particular construction with $R_0=\tfrac{1}{2}$ is
$\widehat \kappa_{\rm rad}={\rm rect} \ast \varphi$, where $\varphi \in \Spc S(\R)$ is a symmetric, non-negative test function (to avoid oscillations) with
${\rm support}(\varphi) \subseteq  [-\tfrac{1}{2},\tfrac{1}{2}]$ and $\int_\R \varphi(x)\dint x=1$. Next, we observe that
\begin{align}
\langle m_{\V k},(-1)^{|\V n|}\partial^{\V n}\kappa_{\rm iso} \rangle=\langle \partial^{\V n}m_{\V k},\kappa_{\rm iso} \rangle=
\begin{cases}
\langle m_{\V k-\V n},\kappa_{\rm iso} \rangle,& \V k-\V n\ge \V 0\\
\langle 0,\kappa_{\rm iso} \rangle=0, & \mbox{otherwise.}
\end{cases}
\end{align}
We evaluate the duality product for the case $\V m=(\V k- \V n) \ge \V 0$ in the Fourier domain  as
\begin{align}
\langle m_{\V m},\kappa_{\rm iso} \rangle&=\frac{1}{(2 \pi)^d} \langle  \Fourier\{ m_{\V m}\},\widehat \kappa_{\rm iso} \rangle \nonumber\\
&=\langle  \frac{\jj^{|\V m|}}{\V m!}\delta^{(\V m)},\widehat \kappa_{\rm iso}\rangle=\frac{-\jj^{|\V m|}}{\V m!} \partial^{\V m} \widehat \kappa_{\rm iso}(\V 0)=\begin{cases}
1,& \V m= \V 0\\
0, & \mbox{otherwise},
\end{cases}
\end{align}
where we have used the relation $\Fourier\{\V x^{\V m}\}(\bw)=(2\pi)^d\; \jj^{|\V m|} \delta^{(\V m)}(\bw)$. Finally, since $\widehat \kappa_{\rm iso}$ is compactly supported, its inverse Fourier transform $\kappa_{\rm iso}$ is an entire function of exponential-type (by the Paley-Wiener theorem). This means that the function $\V x \mapsto \kappa_{\rm iso}(\V x)$ is analytic
with a convergent Taylor series of the form \eqref{Eq:Taylor} for any $\V x, \V x_0 \in \R^d$.
\end{proof}
This allows us to describe the dual space explicitly as
\begin{align}
\label{Eq:DualNullspace}
\Spc P'_{n_0}=\{p_0^\ast=\sum_{|\V k|\le n_0} b^\ast_\V k m^\ast_{\V k}: \|p_0^\ast\|_{\Spc P'}\eqdef \|(b^\ast_\V k)\|_2<\infty\}
\end{align}
where each elements $p_0^\ast$ has a unique representation in terms of its coefficients $(b^\ast_\V k)_{|\V k|\le n_0}$.
We can 
use the dual basis $\{m_\V k^\ast\}$ to specify the projection operator
$\Proj_{\Spc P_{n_0}}: \Spc S'(\R^d) \to \Spc P_{n_0}$ as
\begin{align}
\label{Eq:ProjP}
\Proj_{\Spc P_{n_0}}\{f\}&=\sum_{|\V k|\le n_0} \langle f,m_\V k^\ast  \rangle \; m_\V k,
\end{align}
which is well-defined for any $f \in \Spc S'(\R^d)$ since $m_\V k^\ast \in \Spc S(\R^d)$.

  \subsection{Radon Transform}
 The Radon transform integrates of a function of $\R^d$ over all 
hyperplanes of dimension $(d-1)$. These hyperplanes are indexed over
$\R \times \mathbb{S}^{d-1}$, where $\mathbb{S}^{d-1}=\{\V \xi \in \R^d: \|\V \xi\|_2=1 \}$ is the unit sphere in $\R^d$.
Specifically, the coordinates $\V x$ of a hyperplane associated with an offset $t\in\R$ and a normal vector $\V \xi \in \mathbb{S}^{d-1}$ satisfy 
\begin{align}
\V \xi^\Top \V x=\xi_1x_1+ \dots + \xi_d x_d = t.
\end{align}
The transform is first described for ordinary (test) functions and then extended to tempered distributions by duality.

\subsubsection{Classical Integral Formulation}
The Radon transform of the function $f\in L_1(\R^d)$ is defined as
\begin{align}
\Op R\{ f\}(t, \V \xi)
&=\int_{\R^d}\delta(t-\V \xi^\Top\V x)  f(\V x) \dint \V x,\quad (t,\V \xi) \in \R \times \mathbb{S}^{d-1}. \label{Eq:Radon2}
\end{align}
The adjoint of $\Op R$ is the back-projection operator $\Op R^\ast$.
Its action on $g: \R \times \mathbb{S}^{d-1} \to \R$ yields the function
\begin{align}
\Op R^\ast \{g\}(\V x)=\int_{\mathbb{S}^{d-1}} g(\underbrace{\V \xi^\Top\V x}_{t}, \V \xi)\dint \V \xi, \quad\V x\in \R^d.
\label{Eq:back-projection}
\end{align}
Given the $d$-dimensional Fourier transform $\hat f=\Fourier\{f\}$ 
of  $f \in L_1(\R^d)$, one can calculate $ \Op R \{f\}(\cdot,\V \xi_0)$ for any fixed  $\V \xi_0 \in \mathbb{S}^{d-1}$ through
\begin{align}
\Op R\{ f\}(t, \V \xi_0)=\frac{1}{2 \pi} \int_{\R} \hat f(\omega\V \xi_0) \ee^{\jj \omega t} \dint \omega= \Fourier^{-1}_{\omega \to t} \{  \hat f(\omega\V \xi_0) \}\{t\}.
\label{Eq:CentralSliceTheo}
\end{align}
In other words, the restriction of $\hat f: \R^d \to \C$ along the ray $\{\bw=\omega \V \xi_0: \omega \in \R\}$ is equal to the 1D Fourier transform of $\Op R\{ f\}(\cdot, \V \xi_0)$, a property that is referred to as the {\em Fourier slice theorem}.

%

To describe the functional properties of the Radon transform, one needs the 
 (hyper)spherical (or Radon-domain) counterparts of the spaces described in Section \ref{Sec:Notations} where the Euclidean indexing with $\V x \in \R^d$ is replaced by $(t, \V \xi) \in \R \times  \mathbb{S}^{d-1}$.
The spherical counterpart of $\Spc S(\R^d)$ is $\Spc S(\R \times \mathbb{S}^{d-1})$. Correspondingly, an element $g \in \Spc S'(\R \times \mathbb{S}^{d-1})$ is a continuous linear functional on $\Spc S(\R \times \mathbb{S}^{d-1})$ whose action on the test function $\phi(t,\V \xi)$ is represented by the duality product $g: \phi \mapsto \langle g,\phi \rangle_{\rm Rad}$. When $g$ can be identified as an ordinary function $g: (t,\V \xi) \mapsto \R$, one has that
\begin{align}
\langle g,\phi\rangle_{\rm Rad} = \int_{\mathbb{S}^{d-1}} \int_{\R} g(t, \V \xi) \phi(t, \V \xi) \dint t \dint \V \xi,
\end{align}
where $\dint \V \xi$ stands for a surface element on $\mathbb{S}^{d-1}$ with $\|\V \xi\|_2=1$. For instance, for $d=2$, we 
parameterize $\mathbb{S}^{1}$ 
by setting $\V \xi=(\cos \theta, \sin \theta)$ with $\dint \V \xi=\dint \theta$ for $\theta\in [0,2 \pi]$, which then yields
\begin{align}
\langle g,\phi \rangle_{\rm Rad} = \int_{0}^{2 \pi} \int_{\R} g(t, \theta) \phi(t, \theta)\dint t  \dint \theta.
\end{align}
Such explicit representations are also available in higher dimensions using hyper-spherical polar coordinates.
Of special importance to us is the translated and rotated hyper-spherical Dirac distribution 
$\delta_{(t_0, \V \xi_0)}=\delta(\cdot-t_0)\delta(\cdot-\V \xi_0)\in \Spc S'(\R \times \mathbb{S}^{d-1})$, which is defined as
$\langle \delta_{\V z_0},\phi \rangle_{\rm Rad}\eqdef \phi(\V z_0)$ for all
$\phi \in  \Spc S(\R \times \mathbb{S}^{d-1})$ and any offset $\V z_0=(t_0, \V \xi_0) \in \R \times \mathbb{S}^{d-1}$.
This Dirac impulse, which is separable in the index variables $t$ and $\V \xi$, is included in the Banach space $\Spc M (\R \times \mathbb{S}^{d-1})$ (hyper-spherical Radon measures) with the property that
$\|\delta_{\V z_0}\|_{\Spc M}=1$.

The key property for analysis is that the Radon transform is continuous on $\Spc S$ and invertible (see \citep{Ludwig1966,Helgason2011,Ramm2020} for details).

\begin{theorem}[Continuity and Invertibility of the Radon Transform on $\Spc S(\R^d)$]
\label{Theo:RadonS0}
The Radon operator $\Op R$ continuously maps $\Spc S(\R^d) \to \Spc S(\R \times \mathbb{S}^{d-1})$. Moreover, $\Op R^\ast \Op K_{\rm rad} \Op R=\Op K \Op R^\ast \Op R=\Op R^\ast \Op R\Op K =\Identity \mbox{ on }\Spc S(\R^d)
$,
where $\Op K=(\Op R^\ast \Op R)^{-1}=c_d(-\Delta)^{\tfrac{d-1}{2}}$ with $c_d=\frac{1}{2(2\pi)^{d-1}}$ is the so-called ``filtering'' operator, 
and is $ \Op K_{\rm rad}$ its one-dimensional radial counterpart that acts along the Radon-domain variable $t$. These filtering operators are characterized by their frequency response 
$\widehat K(\bw)=c_d\|\bw\|^{d-1}$ and
$\widehat K_{\rm rad}(\omega)=c_d |\omega|^{d-1}$. \end{theorem}

The image of $\Spc S(\R^d)$ through the Radon transform is the space $\Spc S_{\rm Rad}\eqdef\Op R\big(\Spc S(\R^d)\big)$: a subset of the space of hyper-spherical test functions $\Spc S(\R \times   \mathbb{S}^{d-1})$ that can be characterized explicitly \citep{Gelfand1966,Helgason2011,Ludwig1966}. In particular, a function $\phi=\Op R\{\varphi\} \in \Spc S_{\rm Rad}$ must be symmetric and such as  $\int_\R \Op R \{\varphi\}(t,\V \xi)\dint t=\int_{\R^d}\varphi(\V x)\dint \V x$ for all $\V \xi \in \mathbb{S}^{d-1}$ (see Theorem \ref{Theo:RangeRadon} for details). 
While Theorem \ref{Theo:RadonS0}
implies that $\Op R$ is invertible on $\Spc S_{\rm Rad}$, there is also a stronger version of this property for $\Spc S_{\rm Rad}$ equipped with the Schwartz-Fréchet topology inherited from $\Spc S(\R \times   \mathbb{S}^{d-1})$ \citep[p. 60]{Helgason2011} \citep{Hertle1983}.

\begin{theorem}
\label{Theo:Sinvert}
The operator $\Op R: \Spc S(\R^d) \to \Spc S_{\rm Rad}$ is a continuous bijection, with a continuous inverse given by $\Op R^{-1}=(\Op R^\ast\Op K_{\rm rad}): \Spc S_{\rm Rad}\to \Spc S(\R^d)$.
\end{theorem}
The bottom line is that the classical Radon transform is a homeomorphism $\Spc S(\R^d) \to \Spc S_{\rm Rad}$.

\subsubsection{Distributional Extension}
\label{Sec:RadonDistributions}
To extend the framework to distributions, one proceeds by duality. For instance,
by invoking the property that 
 $\Op R^\ast \Op K_{\rm rad}\Op R=\Identity$
on $\Spc S(\R^d)$ and noting that 
$\Op K_{\rm rad}$
is self-adjoint, we make the manipulation
\begin{align}
\forall \varphi \in \Spc S(\R^d),\quad 
\langle f,\varphi \rangle&=\langle f,\Op R^\ast \Op K_{\rm rad}\Op R\{\varphi\} \rangle
\nonumber\\
&= \langle \Op K_{\rm rad}\Op R\{f\},\Op R\{ \varphi\}\rangle_{\rm Rad}
=\langle \Op K_{\rm rad}\Op R\{f\}, \phi\rangle_{\rm Rad},\label{Eq:dualFilRad}
\end{align}
with $\phi=\Op R\{\varphi\} \in \Spc S_{\rm Rad}$ 
and $\varphi=\Op R^\ast\Op K_{\rm rad} \{\phi\}$. Eq.\ \eqref{Eq:dualFilRad} is valid in the classical sense for $f \in L_1(\R^d)$. Otherwise, it is used as definition to extend the scope of the operator for $f\in \Spc S'(\R^d)$.

\begin{definition}
\label{Def:GeneralizedRadon}
The distribution $g
\in \Spc S'(\R \times   \mathbb{S}^{d-1})$ (respectively, $g=\Op R\{f\} \in \big( \Op K_{\rm rad}\Op R\big( \Spc S)\big)'$ to ensure unicity)
is a formal Radon transform of $f \in \Spc S'(\R^d)$ if
\begin{align}
\forall \phi \in \Op K_{\rm rad}\Op R \big( \Spc S(\R^d)\big):\quad \langle g,\phi \rangle_{\rm Rad}
=\langle f, \Op R^\ast \{\phi\} \rangle.
\label{Eq:RDist}
\end{align}
Likewise, $\tilde g
\in \Spc S'(\R \times   \mathbb{S}^{d-1})$ (respectively, $\tilde g=\Op K_{\rm rad}\Op R\{f\} \in \Spc S_{\rm Rad}'$ to ensure unicity) is a formal filtered projection  of $f \in \Spc S'(\R^d)$ if 
\begin{align}
\forall \phi \in \Spc S_{\rm Rad}: \quad \langle \tilde g,\phi \rangle_{\rm Rad}=\langle f, \Op R^\ast\Op K_{\rm rad} \{\phi\} \rangle.
\label{Eq:KRDist}
\end{align}
Finally, the tempered distribution $f=\Op R^\ast \{g\} \in \Spc S'(\R^d)$ is the back-projection of $g
\in {\Spc S'(\R \times   \mathbb{S}^{d-1})}$ 
if 
\begin{align}
\forall \varphi \in \Spc S(\R^d): \quad \langle \Op R^\ast\{ g\}, \varphi \rangle
=\langle g, \Op R \{\varphi\}\rangle_{\rm Rad}. \label{Eq:RadjDis}
\end{align}
\end{definition}

While \eqref{Eq:RadjDis} yields a unique definition of the linear operator $\Op R^\ast: \Spc S'(\R \times   \mathbb{S}^{d-1}) \to \Spc S'(\R^d)$, the establishement of a proper definition of the 
extended operators $\Op R$ and $(\Op K_{\rm rad}\Op R)$ is trickier because there are always infinitely many distributions $g,\tilde g \in \Spc S'(\R \times   \mathbb{S}^{d-1})$ that satisfy \eqref{Eq:RDist} 
or \eqref{Eq:KRDist}.
The root of the problem is that the Radon transform $\Op R: \Spc S(\R^d) \to \Spc S(\R \times \mathbb{S}^{d-1})$ is not surjective:
Since the test functions $\phi$ in \eqref{Eq:RDist} and \eqref{Eq:KRDist} do not span the whole space $\Spc S(\R \times   \mathbb{S}^{d-1})$, they are unable to separate out
the components that are in the 
null space of $\Op R^{\ast}$, which is huge\footnote{In addition to the odd distributions, ${\rm ker}(\Op R^\ast)=\{h \in \Spc S'(\R \times \mathbb{S}^{d-1}): \Op R^\ast\{h\}=0\}$ is formed of 
the set of rapidly decreasing distributions that are orthogonal to functions of the form $\varphi_m(t,\V \xi)=P_{m-1}(t) Y_m(\V \xi)$ where $P_{m-1}$ is a univariate polynomial of degree less than $m$ and $Y_m$ a spherical harmonic of degree $m$ \citep{Ludwig1966}. }.
To illustrate this ambiguity, we consider the Dirac ridge $\delta(\V \xi_0^\Top\V x - t_0) \in \Spc S'(\R^d)$ and refer to Definition \eqref{Eq:Radon2} of the Radon transform to deduce that, for all $\phi=\Op R\{\varphi\} \in \Spc S_{\rm Rad}$ with $\varphi \in \Spc S(\R^d)$,
\begin{align}
\langle \delta(\V \xi_0^\Top\cdot - t_0),\Op R^\ast \Op K_{\rm rad}\{\phi\}\rangle&=\langle \delta(\V \xi_0^\Top\cdot - t_0), \overbrace{\Op R^\ast\Op K_{\rm rad}\Op R}^{\Identity}\{\varphi\}\rangle \nonumber \\
&= \int_{\R^d}\delta(\V \xi_0^\Top\V x-t_0) \varphi(\V x)  \dint \V x
=\Op R\{\varphi\}(-\V z_0)=\langle\delta_{-\V z_0},\phi\rangle_{\rm Rad}
\label{Eq:DiracRad}
\end{align}
where $\V z_0=(t_0,\V \xi_0)$, which shows that the Dirac impulse $\delta_{-\V z_0}$ is a formal filtered projection of $\delta(\V \xi_0^\Top\V x - t_0)$. Moreover, since 
$\delta(\V \xi_0^\Top\V x - t_0)=\delta(-\V \xi_0^\Top\V x +t_0)$, the same holds true for $\delta_{\V z_0}$, as well as for
$\frac{1}{2} \big(\delta_{\V z_0}+\delta_{-\V z_0}\big)$. In fact, there is an infinity of potential ``formal'' solutions, which is consistent
with the lack of injectivity of $\Op R^\ast: \Spc S'(\R \times \mathbb{S}^{d-1}) \to \Spc S'(\R^d)$. 
%

Classically, one obtains a unique characterization of $g=\Op R\{f\}$ for $f \in \Spc S'(\R^d)$ by restricting the scope of \eqref{Eq:RDist} 
to $g \in  \big(\Op K_{\rm rad}\Op R(\Spc S)\big)'$ \citep{Ludwig1966}. 
Likewise, the description of $\tilde g=\Op K_{\rm rad}\Op R\{f\}$ in \eqref{Eq:KRDist} is unique if it is restricted to $\tilde g\in \big(\Op R(\Spc S)\big)'=\Spc S'_{\rm Rad}$, which then yields a proper definition of
$(\Op K_{\rm rad}\Op R): \Spc S'(\R^d) \to \Spc S'_{\rm Rad}$. 

It turns out that this mechanism is equivalent to factoring out the null space of $\Op R^\ast: \Spc S'(\R \times \mathbb{S}^{d-1}) \to \Spc S'(\R^d)$.
Indeed, since the extended back-projection operator specified by \eqref{Eq:RadjDis} is continuous $\Spc S'(\R \times \mathbb{S}^{d-1})  \to \Spc S'(\R^d)$, its null space
\begin{align}
\label{Eq:NullSpace}
\Spc N_{\Op R^\ast}=\{g \in \Spc S'(\R \times \mathbb{S}^{d-1}): \Op R^\ast\{g\}=0\quad \Leftrightarrow \quad \langle g, \phi  \rangle_{\rm Rad}=0, \forall \phi \in \Spc S_{\rm Rad}\}
\end{align}
is a closed subspace of  $\Spc S'(\R \times \mathbb{S}^{d-1})$. Accordingly, we can identify $\Spc S'_{\rm Rad}$ 
as the abstract quotient space $\Spc S'(\R \times \mathbb{S}^{d-1})/\Spc N_{\Op R^\ast}$. In other words, if we find a hyper-spherical distribution $g_0\in \Spc S'(\R \times \mathbb{S}^{d-1})$ such that
\eqref{Eq:KRDist} is met for a given $f \in \Spc S'(\R^d)$, then, strictly speaking, $\Op K_{\rm rad}\Op R\{f\} \in \Spc S'_{\rm Rad}$ is the equivalence class (or coset) given by
\begin{align}
\label{Eq:FProjEquivalenceClass}
\Op K_{\rm rad}\Op R\{f\}=[g_0]=\{g_0 + h: h \in \Spc N_{\Op R^\ast}\}.
\end{align}
The members of $[g_0]$ are interchangeable---we refer to them as ``formal'' filtered projections of $f$ to remind us of this lack of unicity.

Based on those definitions, one obtains the classical result on the invertibility of the (filtered) Radon transform on $\Spc S'(\R^d)$ \citep{Ludwig1966}, which is the dual of 
Theorem \ref{Theo:Sinvert}.

\begin{theorem}[Invertibility of the Radon Transform on $\Spc S'(\R^d)$]
\label{Theo:InvertRadonDist}
It holds that $\Op R^\ast \Op K_{\rm rad} \Op R
=\Identity$ on $\Spc S'(\R^d)$.
Moreover, the filtered Radon transform $\Op K_{\rm rad}\Op R: \Spc S'(\R^d) \to \Spc S'_{\rm Rad}$ is bicontinuous and one-to-one, with its continuous inverse $(\Op K_{\rm rad}\Op R)^{-1}$ being the back-projection operation $\Op R^\ast: \Spc S'_{\rm Rad} \to 
\Spc S'(\R^d)$.
\end{theorem}

The distributional extension of the Radon transform inherits most of the properties of the ``classical'' operator defined in \eqref{Eq:Radon2}. 
Of special relevance to us is the quasi-commutativity of $\Op R$ with convolution, also known as the {\em intertwining property}. Specifically, let $h,f \in \Spc S'(\R^d)$ be two distributions whose convolution $h \ast f$ is well defined in $\Spc S'(\R^d)$. Then,
\begin{align}
\Op R\{ h \ast f\}=\Op R\{ h\} \ast_t  \Op R\{ f\}\end{align}
where the symbol ``$\ast_t$" denotes a convolution along the radial variable $t$; that is, $(u \ast_t g)(t,\V \xi)=\langle u(\cdot,\V \xi),g(t-\cdot,\V \xi) \rangle$. In particular, when $h=\Lop\{\delta\}$ is the (isotropic) impulse response of an LSI operator whose frequency response $\widehat L(\bw)=\widehat L_{\rm rad}(\|\bw\|)$ is purely radial, we get that
\begin{align}
\label{Eq:CommutRad1}\Op R\{ h \ast f\}=\Op R\Lop\{f\}=\Lop_{\rm rad}\Op R\{f\},
\end{align}
where $\Lop_{\rm rad}$ is the ``$\ast_t$" convolution operator whose 1D frequency response is $\widehat L_{\rm rad}(\omega)$. Likewise, by duality, for $g \in \Spc S'(\R \times   \mathbb{S}^{d-1})$ we have that
\begin{align}
\label{Eq:CommutRad2}
\Lop \Op R^\ast\{g\}=\Op R^\ast \Lop_{\rm rad}\{g\}
\end{align}
under the implicit assumption that $\Lop \{\Op R^\ast g\}$ and $\Lop_{\rm rad}\{g\}$ are well-defined distributions. 
By taking inspiration from Theorem \ref{Theo:RadonS0}, we may then use Relations \eqref{Eq:CommutRad1} and \eqref{Eq:CommutRad2} for $\Lop=\Op K=(\Op R^\ast \Op R)^{-1}$ to show that $\Op R^\ast\Op K_{\rm rad} \Op R\{f\}=\Op R^\ast \Op R\Op K\{f\}=\Op K\Op R^\ast\Op R\{f\}=f$ for a broad class of distributions. While the first form is valid for all $f \in \Spc S'(\R^d)$ (see Theorem \ref{Theo:InvertRadonDist}), there is a slight restriction with the
second (resp., third), which requires that $\Op K\{f\}$ \big(resp., $\Op K\{g\}$ with $g=\Op R^\ast\Op R\{f\} \in \Spc S'(\R^d)$\big) be well-defined in $\Spc S'(\R^d)$. While the latter condition is always met when $d$ is odd, it can fail in even dimensions for distributions (e.g., polynomials) whose Fourier transform is singular at the origin\footnote{For $d=2n$ even, $\widehat K(\bw) \propto \|\bw\|^{2n-1}$ which is $C^\infty$ everywhere except at the origin, where it is only $C^{2n-2}$, meaning that $\Op K$ can only properly handle (and annihilate) polynomials up to degree $(2n-2)$.}.

The Fourier-slice theorem expressed by \eqref{Eq:CentralSliceTheo} also yields a unique (Fourier-based) characterization of $\Op R \{f\}$, which remains valid for $f\in \Spc S'(\R^d)$ \citep{Ramm2020}. 
It is especially helpful 
when the underlying function or distribution is isotropic. 
An isotropic function $\rho_{\rm iso}: \R^d \to \R$ is characterized 
by its radial profile $\rho: \R_{\ge0} \to \R$, with
$\rho_{\rm iso}(\V x)=\rho(\|\bx\|)$. 
%
The frequency-domain counterpart of this characterization is 
$\widehat \rho_{\rm iso}(\bw)=\widehat \rho_{\rm rad}(\|\bw\|)$ where the radial frequency profile 
can be computed as
\begin{align}
\widehat \rho_{\rm rad}(\omega)= \frac{(2\pi)^{d/2}}{|\omega|^{d/2-1} }\int_{0}^{+\infty} \rho(t) t^{d/2-1} J_{d/2-1}(\omega t) t \dint t,
\end{align}
where  $J_{\nu}$ is the Bessel function of the first kind of order $\nu$.
\begin{proposition} [Radon Transform of Isotropic Distributions] 
\label{Prop:IsoRad}
Let $\rho_{\rm iso}$ be an isotropic distribution whose radial frequency profile is $\widehat \rho_{\rm rad}(\omega)$. Then,
\begin{align}
\Op R\{\rho_{\rm iso}(\cdot -\V x_0)\}(t,\V \xi)&=\rho_{\rm rad}(t-\V \xi^\Top \V x_0) \\
\Op K_{\rm rad}\Op R\{ \rho_{\rm iso}(\cdot -\V x_0)\}(t,\V \xi\}&=\tilde \rho_{\rm rad}(t-\V \xi^\Top \V x_0)\\
\Op R\{\partial^{\V m}\rho_{\rm iso}\}(t,\V \xi)&=\V \xi^{\V m} \Op D^{|\V m|}\{\rho_{\rm rad}\}(t)
\label{Eq:RadDeriv}
\end{align} 
with $\rho_{\rm rad}(t)=\Fourier^{-1}\{\widehat \rho_{\rm rad}(\omega)\}(t)$ and 
$\tilde\rho_{\rm rad}(t)=\tfrac{1}{2(2\pi)^{d-1}}\Fourier^{-1}\{|\omega|^{d-1}\widehat\rho_{\rm rad}(\omega)\}(t)$. 

\end{proposition}
\begin{proof} These identities are direct consequences of the
Fourier-slice theorem. For instance, by setting $\bw =\omega \V \xi$ in the Fourier transform of $\partial^{\V m}\rho_{\rm iso}$, we get that
\begin{align}
\widehat{\partial^{\V m}\rho_{\rm iso}}(\omega \V \xi)
=(\jj \omega\V \xi)^{\V m}\widehat \rho_{\rm rad}(\omega)=\V \xi^{\V m} (\jj \omega)^{|\V m|} \widehat \rho_{\rm rad}(\omega)
\end{align}
which, upon taking the inverse 1D Fourier transform, yields \eqref{Eq:RadDeriv}. 

\end{proof}

%
%
Let us note that both $\rho_{\rm rad}$ and  $\tilde\rho_{\rm rad}$, as inverse Fourier transform of a real-valued function, are symmetric, which is consistent with the symmetry of the Radon transform and its filtered version.

\section{Radon-Compatible Banach Spaces}
\label{Sec:back-projection}
The distribution formalism associated with Theorem \ref{Theo:InvertRadonDist}
is extremely powerful. It provides a definition of the (filtered) Radon transform of any tempered distribution $f \in \Spc S'(\R^d)$, along with a way to invert it via the back-projection operator $\Op R^\ast: \Spc S'_{\rm Rad} \to \Spc S'(\R^d)$.
While this is conceptually very satisfying, it is only moderately helpful for our purpose because the members of $\Spc S'_{\rm Rad}$ are ``abstract'' equivalence classes of distributions. In fact, the investigation of functional-optimization problems with Radon-domain regularization requires some Banach counterpart of Theorems \ref{Theo:Sinvert} and \ref{Theo:InvertRadonDist}, where both $\Op R\{f\}$ and $\Op K_{\rm Rad}\Op R\{f\}$ have concrete representations as hyper-spherical functions or measures. In particular, we are interested in identifying specific Radon-domain Banach spaces---for instance, an appropriate subspace of hyper-spherical measures---over which the back-projection operator $\Op R^\ast$ is guaranteed to be invertible.
\label{Sec:ComplementedSpaces}
To that end, we pick a ``parent'' hyper-spherical Banach space $\Spc X=(\Spc X, \|\cdot\|_{\Spc X})$ such that
$\Spc S(\R \times   \mathbb{S}^{d-1}) \embedD \Spc X \embedD \Spc S'(\R \times   \mathbb{S}^{d-1})$.
This dense embedding hypothesis has several implications.
\begin{enumerate}
\item The Banach space $\Spc X$ is the completion of $\Spc S(\R \times   \mathbb{S}^{d-1})$ in the $\|\cdot\|_{\Spc X}$ norm; i.e., 
\begin{align}
\label{Eq:Xspace}
\Spc X=\overline{\big(\Spc S(\R \times   \mathbb{S}^{d-1}), \|\cdot\|_{\Spc X}\big)}.
\end{align}
\item The dual space $\Spc X'\embedC \Spc S'(\R \times   \mathbb{S}^{d-1})$ is equipped with the norm
\begin{align}
\label{Eq:Dualnorm}
\|g\|_{\Spc X'}=\sup_{\phi \in \Spc X:\; \|\phi\|_{\Spc X}\le 1} \langle g, \phi \rangle
=\sup_{\phi \in \Spc S(\R \times   \mathbb{S}^{d-1}):\; \|\phi\|_{\Spc X}\le 1} \langle g, \phi \rangle,
\end{align}
where the restriction of $\phi \in \Spc S(\R \times   \mathbb{S}^{d-1})$ on the rightmost side of \eqref{Eq:Dualnorm} is justified by the denseness of
$\Spc S(\R \times   \mathbb{S}^{d-1})$ in $\Spc X$.
\item The definition of $\|g\|_{\Spc X'}$ found in the rightmost side of \eqref{Eq:Dualnorm} is valid for 
any distribution $g\in \Spc S'(\R \times   \mathbb{S}^{d-1})$ with $\|g\|_{\Spc X'}=\infty$ for $g \notin \Spc X'$. Accordingly, we can specify the topological dual of $\Spc X'$ as
\begin{align}
\label{Eq:DualX}
\Spc X'=\big\{g \in  \Spc S'(\R \times   \mathbb{S}^{d-1}): \|g\|_{\Spc X'}< \infty\big\}.
\end{align}
\end{enumerate}
Prototypical examples where those properties are met are $(\Spc X, \Spc X')=\big(L_p(\R \times   \mathbb{S}^{d-1}),L_q(\R \times   \mathbb{S}^{d-1})\big)$ with $p\in[1,\infty)$ and $q=p/(p-1)$ (conjugate exponent), as well as
$(\Spc X, \Spc X')=\big(C_0(\R \times   \mathbb{S}^{d-1}),\Spc M(\R \times   \mathbb{S}^{d-1})\big)$ for $p=\infty$.

Likewise, by considering the dual pair $(\Spc S_{\rm Rad}, \Spc S'_{\rm Rad})$, we specify our Radon-compatible Banach subspaces
\begin{align}
\Spc X_{\rm Rad}&=\overline{(\Spc S_{\rm Rad},\|\cdot\|_{\Spc X})}
\label{Eq:Xrad}\\
\Spc X'_{\rm Rad}&=\big(\Spc X_{\rm Rad}\big)'=\big\{g \in  \Spc S'_{\rm Rad}: \|g\|_{\Spc X'_{\rm Rad}}< \infty\big\},
\label{Eq:XradP}
\end{align}
where the underlying dual norms have a definition that is analogous to \eqref{Eq:Dualnorm} with
$\Spc S(\R \times   \mathbb{S}^{d-1})$ and $\Spc X$ being instanciated by 
$\Spc S_{\rm Rad}$ and $\Spc X_{\rm Rad}$.
We now show that $\Op R^\ast$ (resp., $\Op K_{\rm rad}\Op R^\ast$) is invertible on $\Spc X'_{\rm Rad}$ (resp., on $\Spc X_{\rm Rad}$), which is the main theoretical contribution of this work.
%

Since $\Op R: \Spc S(\R^d) \to \Spc S_{\rm Rad}$ is a homeomorphism and the $\|\cdot\|_{\Spc X}$-norm  is continuous in the Fréchet topology of $\Spc S_{\rm Rad}$ \big(resp., of $ \Spc S(\R \times \mathbb{S}^{d-1})$\big), we can consider the normed space $(\Spc S(\R^d),\|\cdot\|_{\Spc Y})$ with
$\|\varphi\|_{\Spc Y}\eqdef\|\Op R\{\varphi\}\|_{\Spc X}$ and identify the operator
$\Op R: (\Spc S(\R^d),\|\cdot\|_{\Spc Y}) \to (\Spc S_{\rm Rad},\|\cdot\|_{\Spc X})$
as an isometry. We then invoke the BLT theorem \citep{Reed1980} to specify the unique
extension $\Op R: \Spc Y \to \overline{(\Spc S_{\rm Rad},\|\cdot\|_{\Spc X})}=\Spc X_{\rm Rad}$, where $\Spc Y$ is a Banach space 
isometric to $\Spc X_{\rm Rad}$. More precisely, we have that
\begin{align}
\Spc Y=\overline{(\Spc S(\R^d),\|\cdot\|_{\Spc Y})}=\{\Op R^{-1}\{g\}: g \in \Spc X_{\rm Rad}\}
\end{align}
with $\Op R^{-1}=\Op R^\ast\Op K_{\rm rad}$ on $\Spc S_{\rm Rad}$ and, by extension, on $\Spc X_{\rm Rad}$. This then leads to the following characterization.

\begin{theorem}[Radon-compatible Banach spaces]
\label{Theo:Complementedspaces}
Let $(\Spc X_{\rm Rad},\Spc X'_{\rm Rad})$ be a dual pair of hyper-spherical Banach spaces induced by a norm $\|\cdot\|_{\Spc X}$ and specified by \eqref{Eq:Xrad} and \eqref{Eq:XradP}. Then,
the following properties hold.
\begin{enumerate}
\item The back-projection $\Op R^\ast: \Spc X'_{\rm Rad} \to  \Spc S'(\R^d)$ is injective and
$\Op K_{\rm rad}\Op R\Op R^\ast=\Identity$ on $\Spc X'_{\rm Rad}$.
\item The filtered back-projection $\Op R^\ast\Op K_{\rm rad}: \Spc X_{\rm Rad} \to  \Spc Y$ is an isometric bijection, with $\Op R\Op R^\ast\Op K_{\rm rad}=\Identity$ on $\Spc X_{\rm Rad}$.
\item The corresponding ``range'' spaces $\Spc Y=\Op R^\ast\Op K_{\rm rad}(\Spc X_{\rm Rad})$ and
$\Spc Y'=\Op R^\ast(\Spc X'_{\rm Rad})$ from a dual Banach pair that is isomorphic to $(\Spc X_{\rm Rad},\Spc X'_{\rm Rad} )$.
\end{enumerate}
Moreover, if there exists a complementary Banach space $\Spc X^{\rm c}_{\rm Rad}$ such that $
\Spc X=\Spc X_{\rm Rad}\oplus\Spc X^{\rm c}_{\rm Rad}$, then additional properties hold.
\begin{enumerate}
\item The dual space is decomposable as $\Spc X'=\Spc X'_{\rm Rad}\oplus(\Spc X^{\rm c}_{\rm Rad})'$.
\item The dual complement $(\Spc X^{\rm c}_{\rm Rad})'$ is the null space of $\Op R^\ast: \Spc X' \to  \Spc Y'$.
\item The complement space $\Spc X^{\rm c}_{\rm Rad}$ is the null space of $\Op R^\ast\Op K_{\rm rad}: \Spc X \to  \Spc Y$.

\item The operators $\Op P_{\rm Rad}=\Op R\Op R^\ast\Op K_{\rm rad}: \Spc X \to \Spc X_{\rm Rad}$ and $\Op P^\ast_{\rm Rad}=\Op K_{\rm rad}\Op R\Op R^\ast: \Spc X' \to \Spc X_{\rm Rad}$ form an adjoint pair of continuous projectors
with $\Op P_{\rm Rad}\big(\Spc X\big)=\Spc X_{\rm Rad}$ and $\Op P^\ast_{\rm Rad}\big(\Spc X'\big)=\Spc X_{\rm Rad}'$.
\end{enumerate}

\end{theorem}
\begin{proof}\\ First part: On the one hand, the existence of an extended isometry $\Op R: \Spc Y \to \Spc X_{\rm Rad}$ implies
the continuity of $\Op R^\ast:\Spc X'_{\rm Rad} \to \Spc Y'$ with $\|\Op R\|=\|\Op R^\ast\|=1$ and $\Op R^\ast(\Spc X'_{\rm Rad})\subseteq\Spc Y'$. On the other hand, Theorem \ref{Theo:InvertRadonDist} implies that the back-projection operator 
$\Op R^\ast: \Spc S'_{\rm Rad} \to  \Spc S'(\R^d)$ defined by \eqref{Eq:RadjDis} is injective with a continuous inverse on $\Spc S'_{\rm Rad}$ that is given by $(\Op R^\ast)^{-1}=\Op K_{\rm rad}\Op R$. Let $\Spc V=\Op R^\ast\big(\Spc X'_{\rm Rad}\big)$ be the range of the restricted operator $\Op R^\ast\vert_{\Spc X'_{\rm Rad}}$ with $\Spc X'_{\rm Rad} \embedC \Spc S_{\rm Rad}'$ (by construction).
Then, $\Op R^\ast: \Spc X'_{\rm Rad} \to \Spc V$ is injective and invertible with
$(\Op R^\ast)^{-1}=\Op K_{\rm rad} \Op R: \Spc V \to \Spc X'_{\rm Rad}$. This implies that $\Spc V=\Op R^\ast\big(\Spc X'_{\rm Rad}\big)$ is a Banach space that is isomorphic to 
$\Spc X'_{\rm Rad}$ (see \citep[Proposition 1]{Unser2022}).

Injectivity of $\Op R^\ast\Op K_{\rm rad}$ on $\Spc X_{\rm Rad}$: The continuity of $\Op K_{\rm rad} \Op R: \Spc V \to \Spc X'_{\rm Rad}$, together with the isometric embedding of a Banach space in its bidual, implies the continuity of
$\Op R^\ast\Op K_{\rm rad}: \Spc X_{\rm Rad} \to \Spc V'$ (isometry). The condition $\Op R^\ast\Op K_{\rm rad}\{\phi\}=0$
can then be restated as $\forall v \in \Spc V: \langle v, \Op R^\ast\Op K_{\rm rad}\{\phi\}\rangle=\langle \Op K_{\rm rad}\Op R\{v\}, \phi\rangle_{\rm Rad}=0$. Since $\Spc X'_{\rm Rad}=\Op K_{\rm rad}\Op R\big(\Spc V\big)$, this is equivalent to
say that $\langle g, \phi\rangle_{\rm Rad}=0$ for all $g \in \Spc X'_{\rm Rad}$, which leads to $\phi=0$ and proves that 
$\Op R^\ast\Op K_{\rm rad}: \Spc X_{\rm Rad} \to  \Spc V'$ is injective. 
Consequently, $\widetilde{\Spc Y}=\Op R^\ast\Op K_{\rm rad}(\Spc X_{\rm Rad}) \embedIso  \Spc V'$ is a Banach space that is isomorphic to
$\Spc X_{\rm Rad}$, while there exists an inverse operator
$(\Op R^\ast\Op K_{\rm rad})^{-1}$ such that $\Spc X_{\rm Rad}=(\Op R^\ast\Op K_{\rm rad})^{-1}\big(\widetilde{\Spc Y}\big)$. By invoking the continuity of
$\Op R: \Spc V' \to \Spc X''_{\rm Rad}$, we then manipulate the duality product as
\begin{align}
\langle g, \phi \rangle_{\rm Rad}=\langle w, \tilde y\rangle=\langle \Op R^\ast\{g\},\tilde y \rangle&= \langle g, \Op R\{\tilde y\} \rangle_{\Spc X'_{\rm Rad}\times \Spc X''_{\rm Rad}} \nonumber\\
&=\langle g,\Op R\Op R^\ast\Op K_{\rm rad}\{\phi\}\rangle_{\rm Rad},\end{align}
for all $(g,\phi) \in \Spc X'_{\rm Rad} \times \Spc X_{\rm Rad}$ with $w=\Op R^\ast\{g\}\in \Spc Y'$ and $\tilde y=\Op R^\ast\Op K_{\rm rad}\{\phi\}$.
This identity proves that $\Op R\Op R^\ast\Op K_{\rm rad}=\Identity$ on $\Spc X_{\rm Rad}$.
Thus, $\Op R^\ast\Op K_{\rm rad}: \Spc X_{\rm Rad} \to \widetilde{\Spc Y}$ is an isometric bijection with $\Spc X_{\rm Rad}=\Op R(\widetilde{\Spc Y})=\Op R(\Spc Y)$, which also leads to the identification $\widetilde{\Spc Y}=\Spc Y$.

Since $(\Spc X_{\rm Rad},\Spc Y)$ 
 and  $(\Spc X'_{\rm Rad},\Spc V)$  
are homeomorphic pairs of Banach spaces with $\Spc V \subseteq \Spc Y'$, we readily deduce that $\Spc V=\Spc Y'$, with the dual pair $(\Spc Y,\Spc Y')$ ultimately being isometrically isomorphic to $(\Spc X_{\rm Rad},\Spc X'_{\rm Rad})$. The corresponding functional picture can then be summarized as follows.

{\em The dual spaces} $(\Spc X'_{\rm Rad},\Spc Y')$: For every $w\in \Spc Y'$, there exists a unique $g=\Op K_{\rm rad}\Op R\{w\}\in \Spc X'_{\rm Rad}$ such that
$w=\Op R^\ast\{g\}$ with $\|w\|_{\Spc Y'}=\|g\|_{\Spc X'_{\rm Rad}}$, and vice versa (i.e., $\Spc X'_{\rm Rad}=\Op K_{\rm rad}\Op R(\Spc Y')$). Another way to put it is that
$\Op K_{\rm rad}\Op R\Op R^\ast=\Identity$ on $\Spc X'_{\rm Rad}$, which is the statement in Item 1. Likewise, we have that $\Op R^\ast\Op K_{\rm rad}\Op R=\Identity$ on $\Spc Y'$.

{\em The predual spaces} $(\Spc X_{\rm Rad},\Spc Y)$: For every $y\in \Spc Y$, there exists a unique $\phi=\Op R\{y\}\in \Spc X_{\rm Rad}$ such that
$y=\Op R^\ast \Op K_{\rm rad}\{\phi\}$ with $\|y\|_{\Spc Y}=\|\phi\|_{\Spc X_{\rm Rad}}$, and vice versa (i.e., $\Spc X_{\rm Rad}=\Op R(\Spc Y)$). This also implies that
$\Op R\Op R^\ast\Op K_{\rm rad}=\Identity$ on $\Spc X_{\rm Rad}$, as stated in Item 2. Likewise, we have that $\Op R^\ast\Op K_{\rm rad}\Op R=\Identity$ on $\Spc Y$.

Second part:
Item 1' follows from the generic property that $(\Spc X_1\oplus\Spc X_2)'=\Spc X'_1\oplus\Spc X'_2$,
where the Banach spaces $\Spc X_1$ and $\Spc X_2$ can be arbitrary. In particular, Item 1' tells us that
$\Spc X'_{\rm Rad}, (\Spc X^{\rm c}_{\rm Rad})' \subseteq \Spc X'$ with the embedding being continuous. Moreover, it implies that $\Spc X_{\rm Rad}$ is the  annihilator of $\Spc X_{\rm Rad}'$ in $\Spc X'$, with
\begin{align}
(\Spc X^{\rm c}_{\rm Rad})' &=\{g \in \Spc X': \langle g, \phi\rangle_{\rm Rad}=0, \forall \phi \in \Spc X_{\rm Rad} \}\nonumber\\
&=\{g \in \Spc X': \langle g, \phi\rangle_{\rm Rad}=0, \forall \phi \in \Spc S_{\rm Rad} \},
\label{Eq:Xcdense}
\end{align}
where the substitution of $\Spc X_{\rm Rad}$ with $\Spc S_{\rm Rad}$ in \eqref{Eq:Xcdense} is legitimate because $\Spc S_{\rm Rad}$ is a dense subspace of $\Spc X_{\rm Rad}$.
Since $\Spc X'\subset \Spc S'(\R \times   \mathbb{S}^{d-1})$, this shows that
$(\Spc X^{\rm c}_{\rm Rad})'\subseteq \Spc N_{\Op R^\ast}$ (Item 2'). It also yields Item 3' by duality. The null-space property implies that $\Op P^\ast_{\rm Rad}\big((\Spc X^{\rm c}_{\rm Rad})'\big)=\{0\}$ and $\Op P_{\rm Rad}\big(\Spc X^{\rm c}_{\rm Rad}\big)=\{0\}$. The final element is the identity in Item 1 (resp., Item 2), which 
ensures that $\Op P^\ast_{\rm Rad}: \Spc X'\to \Spc X'_{\rm Rad}$ (resp., $\Op P_{\rm Rad}: \Spc X\to \Spc X_{\rm Rad}$) is the canonical projector on $\Spc X'_{\rm Rad}$ (resp.,
on $\Spc X_{\rm Rad}$). The existence (unicity) and continuity of the latter is guaranteed for any pair of complemented Banach spaces (see Appendix A).

\end{proof}

The dual direct-sum decomposition in Theorem \ref{Theo:Complementedspaces} has the following corollary:  
The Banach complement $\Spc X^{\rm c}_{\rm Rad}$ (resp., $\Spc X_{\rm Rad}$) is the annihilator of
$\Spc X'_{\rm Rad}$ in $\Spc X$ (resp., of $(\Spc X^{\rm c}_{\rm Rad})'$ in $\Spc X'$), and vice versa. In particular,
$(\Spc X^{\rm c}_{\rm Rad})'\embedC \Spc S'(\R \times \mathbb{S}^{d-1})$
is specified by \eqref{Eq:Xcdense} (as a set), which clearly shows that $(\Spc X^{\rm c}_{\rm Rad})'\subseteq \Spc N_{\Op R^\ast}$ (see \eqref{Eq:NullSpace}).

The existence of the projection operator $\Op P^\ast_{\rm Rad}$ in the second part of Theorem \ref{Theo:Complementedspaces} is very handy because it enables us to convert any ``formal'' filtered projection $\tilde g=\Op K_{\rm rad}\Op R\{f\} \in \Spc X' \embedC \Spc S'(\R \times \mathbb{S}^{d-1})$ (see \eqref{Eq:KRDist} in Definition \ref{Def:GeneralizedRadon}) into a concrete representer $\Op P^\ast_{\rm Rad}\{\tilde{g}\} \in \Spc X'_{\rm Rad}$, which has a unique, unambiguous interpretation.

Since the members of $\Spc S_{\rm Rad}$ must be even (see Theorem \ref{Theo:RangeRadon}), we readily deduce that 
$\Spc X_{\rm Rad}\subseteq \Spc X_{\rm even}=\overline{\big(\Spc S_{\rm even}(\R \times   \mathbb{S}^{d-1}), \|\cdot\|_{\Spc X}\big)}$. In particular, when the inclusion is a set equality  and when $\Spc X=\Spc X_{\rm even}\oplus\Spc X_{\rm odd}$, then $\Spc X'_{\rm Rad}=\Spc X'_{\rm even}$ and $\Op P^\ast_{\rm Rad}=\Op P_{\rm Rad}=\Op P_{\rm even}$, which is the self-adjoint projector that extracts the even part of a function.

\section{Radon-Domain Sampling Functionals}
\label{Sec:RadonDirac}
The canonical sampling functional that acts on continuous functions expressed in hyper-spherical coordinates  is the shifted hyper-spherical Dirac distribution $\delta_{\V z_0}\in \Spc M(\R \times   \mathbb{S}^{d-1})$ with $\V z_0=(t_0,\V \xi_0) \in \R \times   \mathbb{S}^{d-1}$ and $\|\delta_{\V z_0}\|_{\Spc M}=1$, which also happens to be a ``formal'' filtered projection of $\delta(\V \xi_0^\Top\V x- t_0)$. 

We now show how we can describe $[\delta_{\V z_0}] \in \Spc S'_{\rm Rad}$ by its unique representer $e_{\V z_0}: g \mapsto g(\V z_0)$ in $\Spc M_{\rm Rad}=(C_{0,{\rm Rad}})'$.  
We start with a theoretical investigation where $e_{\V z_0}$ is characterized indirectly through its functional properties. We then provide an explicit construction that allows us to identity 
$e_{\V z_0}$ 
as the limit of a normalized Radon-compatible distribution whose unit mass get concentrated at $\V z=\pm \V z_0$.

%

\subsection{Abstract Characterization of Radon-Compatible Diracs}
Since $\Spc S_{\rm Rad} \subseteq \Spc S_{\rm even}(\R \times   \mathbb{S}^{d-1})$ (see Theorem \ref{Theo:RangeRadon}),
we have that $C_{0,\rm Rad}\subseteq C_{0,\rm even}(\R \times   \mathbb{S}^{d-1})$,
which implies that all functions $g\in C_{0,\rm Rad}$ are continuous and even.
This ensures that the evaluation functional
$e_{\V z_0}: g \mapsto \langle e_{\V z_0}, f  \rangle=g(\V z_0)$ is well-defined for any $\V z_0 \in \R \times   \mathbb{S}^{d-1}$.
We now prove that  $e_{\V z_0}$ is a continuous linear functional on 
$C_{0,\rm Rad}$---that is, $e_{\V z_0} \in (C_{0,\rm Rad})'=\Spc M_{\rm Rad}$ (the space of Radon-compatible measures)---and establish its basic properties, which are compabible with those of the Dirac distribution. 
By the same token, we get a characterization of the extreme points of the unit ball in $\Spc M_{\rm Rad}$.

\begin{theorem}[Properties of $e_{\V z_0}$]
\label{Theo:Extreme0}
The Radon-domain functionals $e_{\V z_0}: C_{0,{\rm Rad}} \to \R$
 with $\V z_0=(t_0,\V \xi_0) \in \R \times   \mathbb{S}^{d-1}$ have the following properties.

\begin{enumerate}
\item Definition (sampling at $\V z_0$)
\begin{align}
\label{Eq:RadonDirac1}
\forall\phi \in C_{0, \rm Rad}
:\quad \langle  e_{(t_0, \V \xi_0)},\phi\rangle_{\rm Rad}=\phi(t_0,\V \xi_0).
\end{align}
\item Symmetry: $e_{\V z_0}=e_{-\V z_0}$.
\item  Continuity:
$\displaystyle \| e_{\V z_0}\|=\sup_{\phi \in C_{0, {\rm Rad}}:\  \|\phi\|_{L_\infty}\le 1} \langle e_{\V z_0}, \phi \rangle=1$, which implies that $e_{\V z_0}\in \Spc M_{\rm Rad}.$ 
\item Let $\{\V z_k\}$ be any finite set of distinct points. Then, $\|\sum_{k} a_ke_{\V z_k}\|_{\Spc M_{\rm Rad}}=\sum_{k} |a_k|$.
\item $\Op R^\ast\{e_{(t_0,\V \xi_0)}\}(\V x)=\delta(\V \xi_0^\Top \V x -t_0) \  \Leftrightarrow \  e_{(t_0,\V \xi_0)}=\Op K_{\rm rad}\Op R\{\delta(\V \xi_0^\Top \V x -t_0)\}$ in $\Spc M_{\rm Rad}$.
\item If $e\in {\rm Ext}B_{\Spc M_{\rm Rad}}$, then $e=\pm e_{\V z_k}$ for some $\V z_k \in \R \times   \mathbb{S}^{d-1}$.
\end{enumerate}
\end{theorem}

\begin{proof}
Item 1 is a paraphrasing of the definition of the evaluation functional, while Item 2 follows from the property that $\phi(\V z_0)=\phi(-\V z_0)$ for all $\phi \in C_{0,{\rm Rad}}$.

Continuity: From the definition of the $\Spc M$-norm and the denseness of $\Spc S_{{\rm Rad}}$ in $C _{0,{\rm Rad}}$, we immediately deduce that
\begin{align*}
\|e_{\V z_0}\|_{\Spc M_{\rm Rad}} &=\sup_{\phi \in \Spc S_{{\rm Rad}}:\  \|\phi\|_{L_\infty}\le 1} \phi(\V z_0)
= \sup_{\phi \in \Spc S_{{\rm Rad}}:\   \|\phi\|_{L_\infty}\le 1} \langle \delta_{\V z_0},\phi\rangle 
 \le \|\delta_{\V z_0}\|_{\Spc M}=1
\end{align*}
for all $\V z_0 \in \R \times   \mathbb{S}^{d-1}$. To prove that $\|e_{\V z_0}\|_{\Spc M_{\rm Rad}}=1$, it therefore suffices to exhibit a function $g_{\V x_0} \in \Spc S_{{\rm Rad}}$
with $\|g_{\V x_0}\|_{L_\infty}=1$ such that $\langle g_{\V x_0}, \delta_{\V z_0}\rangle=1$. We take $g_{\V x_0}(t,\V \xi)=\ee^{-\tfrac{1}{2} (t–\V \xi^\Top \V x_0)^2}$ 
with any choice of  $\V x_0\in\R^d$ such that $\V \xi_0^\Top \V x_0=t_0$: With the help of
Proposition \ref{Prop:IsoRad}, we easily verify that $g_{\V x_0}$ is the Radon transform of a standardized Gaussian centered on $\V x_0$, which ensures that $g_{\V x_0}\in \Spc S_{\rm Rad}$.

Now, let $f=\sum_{k=1}^K a_k e_{\V z_k}$, which is such that
$\|f\|_{\Spc M_{\rm Rad}}\le \sum_{k=1}^K |a_k|=\|\V a\|_{\ell_1}$, by the triangle inequality. 
In addition, since the $\V z_k$ are distinct, there exits some $\epsilon_0>0$ such that
$\|\V z_k-\V z_{k'}\|< \epsilon_0$ for all $k\ne k'$.
Again, we shall prove that $\|f\|_{\Spc M_{\rm Rad}}=\|\V a\|_{\ell_1}$ by constructing a conjugate function $f^\ast \in \Spc S_{\rm rad}$ 
with $\|f^\ast\|_{L_\infty}=1$.
This requires the use of more sophisticaded atoms whose localization is adjustable; namely, the functions
$$\phi_{\epsilon, \V x_k}= \frac{1}{ d_{\epsilon}(0,\V e_1)}d_{\epsilon}(t-\V \xi^\Top \V x_k,\M U_{k}\V \xi) \in \Spc S_{\rm rad}$$ 
where $d_{\epsilon}(t,\V \xi)$ is specified by \eqref{Eq:gepsiRad} and where $\V x_k\in \R^d$
and the rotation matrix $\M U_{k}\in \R^{d\times d}$ are chosen such that $\V \xi_k^\Top\V x_k=t_k$ and $\V e_1=\M U_{k}\V \xi_k$. The function
$\phi_{\epsilon, \V x_k}$ with $\epsilon\in(0,1)$ is symmetric, non-negative and bounded by $1$; it achieves its maximum 
at $\V z=\V z_k$ and is decreasing towards zero as $\V z$ moves away from $\pm\V z_k$ with the speed of decay becoming arbitrarily fast as $\epsilon\to 0$ (see Figure \ref{Fig:RadonDirac}). Consequently, one can always find some $\epsilon>0$ such that
$|\phi_{\epsilon, \V x_k}(\V z)|< 1/K$ for all $\V z
\in \R \times   \mathbb{S}^{d-1}$ with $\|\V z\pm\V z_k\|>\epsilon_0$. Then,
$f^\ast=\sum _{k=1}^K {\rm sign}(a_k)\phi_{\epsilon, \V x_k} \in \Spc S_{\rm Rad}$ is such that $\|f^\ast\|_{L_\infty}=1$ and $\langle f^\ast,f\rangle=\sum_{k=1} |a_k|$. The latter implies that $\|f\|_{\Spc M_{\rm Rad}}\ge \|\V a\|_{\ell_1}$, which proves the claim.

Filtered Radon transform: Theorem \ref{Theo:Complementedspaces} with $\Spc X=C_0(\R \times   \mathbb{S}^{d-1})$ ensures that the adjoint pair of operators
$\Op K_{\rm rad}\Op R^\ast: C_{0,\rm Rad} \to \Spc Y$ and
$\Op K_{\rm rad}\Op R: \Spc Y' \to \Spc M_{\rm Rad}$ are continuous (isometries). 
This Banach setting also allows us to specify the corresponding back-projection operator $\Op R^\ast=(\Op K_{\rm rad}\Op R)^{-1}: \Spc M_{\rm Rad}\to \Spc Y'$ by extending the scope of Definition \eqref{Eq:RadjDis} for $\varphi \in 
\Spc Y=\Op R^\ast\Op K_{\rm rad}(C_{0, \rm Rad})$.
We then use the same manipulations as in \eqref{Eq:DiracRad} with $\phi \in C_{0, \rm Rad}$ (resp., $\varphi \in \Spc Y$)
to prove that: (i) $\delta(\V \xi_0^\Top \V x -t_0)=\Op R^\ast\{e_{t_0,\V \xi_0}\}\in \Spc Y'$,
and (ii) $\Op K_{\rm rad}\Op R\{\delta(\V \xi_0^\Top \V x -t_0)\}=e_{t_0,\V \xi_0} \in \Spc M_{\rm Rad}$. 

The abstract interpretation of Items 1 and 2 is that the evaluation functionals on $C_{0,\rm Rad}$ are spanned (with a double covering) by $e_{\V z}$ with $\V z \in \Spc Z\eqdef\R \times   \mathbb{S}^{d-1}$.
Since $C_{0,\rm Rad}$ is a closed subspace of $C_{0}(\Spc Z)$, we can invoke Lemma \ref{Theo:ExtremeSubspace} in the Appendix, which tells us that the extreme points of the unit ball in
$\Spc M_{\rm Rad}=\big(C_{0,\rm Rad}\big)'$ are all of the form 
$\pm e_{\V z}$ for some $\V z \in \Spc Z$.
\end{proof}

\subsection{Constructive Description of Radon-Compatible Dirac}
In complement to the abstract characterization of $e_{\V z_0}$ in Theorem \ref{Theo:Extreme0},
we now describe the underlying distribution concretely.
We consider the $d$-dimensional Gaussian density function
\begin{align}
g_{\epsilon}(\V x)&=\frac{\epsilon^{d-1}}{(2 \pi)^{d/2}}\exp\left(-\frac{1}{2}\big(\frac{x^2_1}{\epsilon^2}+\epsilon^2(x_2^2 + \cdots+x_d^2 )\big)\right) \in \Spc S(\R^d)
\end{align}
whose
 Fourier transform is
\begin{align}
\label{Eq:GaussFourier}
\hat g_{\epsilon}(\V \bw)=\exp\left(-\frac{1}{2}\big(\epsilon^2\omega_1^2+\frac{\omega_2^2 + \cdots+\omega_d^2}{\epsilon^2}\big)\right).
\end{align}
The parameter $\epsilon<1$ controls the degree of ellipticity. 
When $\epsilon$ is small, $g_{\epsilon}(\V x)$ 
gets narrow along the $x_1$ axis, while it spreads out along the other directions. Setting $\bw=\omega\V \xi$, we rewrite \eqref{Eq:GaussFourier} as 
\begin{align}
\label{Eq:depsilon}
\hat g_{\epsilon}(\omega \V \xi)
&=\exp\left(-\frac{\omega^2}{2}\sigma_\epsilon^2(\V \xi)\right) \mbox{ with } \sigma_\epsilon^2(\V \xi)=\epsilon^2\xi_1^2+\frac{\xi_2^2 + \cdots+\xi_d^2}{\epsilon^2}.
\end{align}
From \eqref{Eq:CentralSliceTheo}, we obtain the Radon transform of $g_\epsilon$ as
\begin{align}
\label{Eq:gepsiRad}
d_{\epsilon}(t,\V \xi)= \Op R\{g_{\epsilon}\}(t,\V \xi)&=\frac{1}{\sqrt{2 \pi \sigma_\epsilon^2(\V \xi)} } \exp\left(-\frac{t^2}{2\sigma_\epsilon^2(\V \xi)}\right) \in \Spc S_{\rm Rad},
\end{align}
which is a radial Gaussian with a spherical dependence on the variance. For $0<\epsilon<1$, $d_{\epsilon}(t,\V \xi)$ attains its maximum at $(t,\V \xi)=(0,\pm\V e_1)$.
As $\epsilon$ gets smaller, the maximum increases while the distribution becomes peakier and more and more localized around  $\V z=(0,\V e_1)$. However, the integral of the function is preserved since $ \int_{\mathbb{S}^{d-1}} \int_{\R} d_{\epsilon}(t,\V \xi)\dint t\dint \V \xi= \int_{\mathbb{S}^{d-1}} \dint \V \xi=\frac{2 \pi^{d/2}}{\Gamma(d/2)}$ for any $\epsilon>0$.
This allows us to identify our Radon-compatible sampling functional as
\begin{align}
\label{Eq:ApproxDirac}
e_{(t_0,\V \xi_0)}(t,\V \xi)=
\lim_{\epsilon\to 0+} d_{\epsilon}(t-\V \xi^\Top \V x_0,\M U_{0}\V \xi),
\end{align}
where $\M U_{0}\inR^{d \times d}$ is a rotation matrix such that $\V e_1=\M U_{0}\V \xi_0$
and $\V x_0 \in \R^d$ is such that $\V \xi^\Top \V x_0=t_0$. Examples of such functions for $d=2$, $\V \xi_0=\V e_1=(1,0)$, and $\V x_0=(2,2)$ are shown in Figure \ref{Fig:RadonDirac}. \begin{figure}

\includegraphics[width=12.5cm]{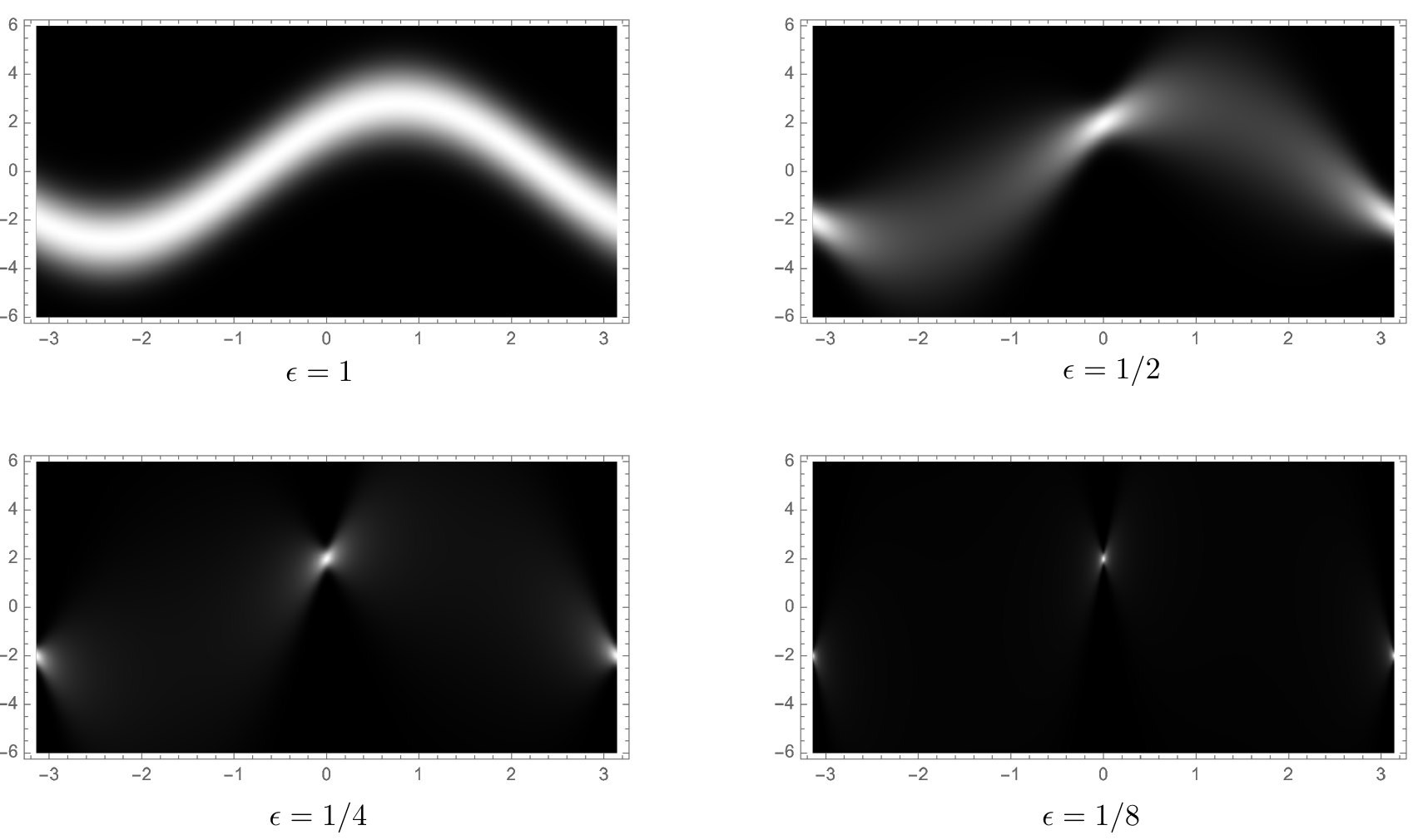}
\caption{Localization effect of the parameter $\epsilon$ 
for the approximation of $e_{(2,\V e_1)}$ for $d=2$, displayed as the sinogram 
of the functions $d_{\epsilon}(t-\V \xi^\Top\V x_0,\theta)$ with $\V \xi=(\cos \theta,\sin \theta)$, $\V x_0=(2,2)$, and
$\epsilon=1,\tfrac{1}{2},\tfrac{1}{4},\tfrac{1}{8}$. 
}
\label{Fig:RadonDirac}

\end{figure}

While this construction reminds us of the description of a Dirac as the limit of a Gaussian distribution whose standard deviation tends to zero, there is one important difference: unlike a conventional Gaussian,
the functions $d_\epsilon$ on the left hand side of \eqref{Eq:ApproxDirac} all satisfy the range conditions of the Radon transform, which are stated in Theorem \ref{Theo:RangeRadon} below.


\section{Ridges 
Revisited}
\label{Sec:Ridges}
A 1D profile along the direction $\V e_1=(1,0,\dots,0)$ is a generalized function of the form
$r_{\V e_1}(\V x)=r(x_1) \times 1$ with $r \in \Spc S'(\R)$. Since the latter is separable, its generalized Fourier transform
is 
\begin{align}
\Fourier\{r_{\V e_1}\}(\bw)=\hat r(\omega_1)\prod_{k=2}^d 2 \pi \delta(\omega_k)=\hat r(\omega_1)(2 \pi)^{d-1}\delta(\omega_2,\dots,\omega_d),
\label{Eq:Ridge0Fourier}
\end{align}
which is a weighted Dirac mass 
localized along the $\omega_1$ axis.
An equivalent formulation of \eqref{Eq:Ridge0Fourier} that involves test functions is
\begin{align}
\label{Eq:Ridge0Fourier2}
\forall \varphi \in \Spc S(\R^d): \langle r_{\V e_1}, \varphi \rangle=\frac{1}{2 \pi} \int_\R \hat r(\omega) \hat \varphi(\omega \V e_1) \dint \omega
\end{align}
where $\hat \varphi \in \Spc S(\R^d)$ is the $d$-dimensional Fourier transform of $\varphi$.
 The argument remains valid when we rotate the coordinate system, which allows us to consider more general ridges of the form
$r_{\V \xi_0}(\V x)\eqdef r(\V \xi_0^\Top \V x)$; that is, 1D profiles along the direction $\V \xi_0 \in \mathbb{S}{d-1}$.
\subsection{Generalized Ridges}
By identifying $\hat \varphi(\omega\V e_1)$ in \eqref{Eq:Ridge0Fourier2} as the 1D Fourier transform of $\Op R \{\varphi(\cdot,\V e_1)\}$ and by substituting $\V e_1$ by $\V \xi_0$, we obtain  the general signal-domain relation:
\begin{align}
\label{Eq:Ridges}
\forall \varphi \in \Spc S(\R^d): \quad \langle r_{\V \xi_0}, \varphi\rangle=\langle r, \Op R\{\varphi\} (\cdot,\V \xi_0)\rangle,\end{align}
which will be referred to as the {\em ridge identity}. Under the assumption that $r$ is a locally integrable function,
we can establish \eqref{Eq:Ridges} by making the change of coordinates $\V y=\M U \V x$, where
$\M U\in \R^{d \times d}$ is any rotation matrix such that $y_1=\V \xi_0^\Top\V x$. We then rewrite the integral explicitly as
\begin{align*}
\int_{\R^d} \varphi(\V x)r_{\V \xi_0}(\V x)\dint \V x&=
\int_{\R^d} \varphi(\V y)r(y_1)\dint y_1 \dots \dint y_d\\
&=
\int_{\R}\underbrace{\left(\int_{\R^{d-1}} \varphi(\V y) \dint y_2 \dots \dint y_d\right)}_{\Op R\varphi(y_1,\V \xi_0)}
 r(y_1)\dint y_1.
\end{align*}
Otherwise, when $r \in\Spc S'(\R)$ has no pointwise interpretation, we simply use \eqref{Eq:Ridges} as definition for the ridge distribution $r_{\V \xi_0}\in \Spc S'(\R^d)$, which is legitimate since $\Op R\{\varphi\} (\cdot,\V \xi_0) \in \Spc S(\R)$.

Note that the special case of \eqref{Eq:Ridges} with $r(t)=\ee^{ - \jj \omega t}$  and
$r_\V \xi(\V x)=\ee^{ - \jj \omega\V \xi^\Top \bx}=\ee^{ - \jj \bw^\Top \bx}$ yields the Fourier-slice theorem \eqref{Eq:CentralSliceTheo}.

The ridge identity \eqref{Eq:Ridges} also helps us delineate the 
range of the Radon transform. Take $r(t)=t^k$ and define
$$\Phi_k(\V \xi)\eqdef\int_{\R} \Op R\{\varphi\}(t,\V \xi)\; t^k \dint t= \int_{\R^d} \varphi(\V x) (\V \xi^\Top\V x)^k \dint \V x
=\sum_{|\bk|=k} a_{\bk} \V \xi^\bk$$ with $a_{\bk}=k!\int_{\R^d} \frac{\V x^\bk}{\bk!}  \varphi(\V x)\dint \V x$.
\begin{theorem}[\citet{Gelfand1966,Helgason2011,Ludwig1966}]
\label{Theo:RangeRadon} A hyper-spherical test function $\phi \in \Spc S(\R \times   \mathbb{S}^{d-1})$ is a valid Radon transform in the sense that $\phi\in \Spc S_{\rm Rad}=\{\Op R\{ \varphi\}: \varphi \in \Spc S(\R^d)\}$ if and only if
\begin{enumerate}
\item Evenness: $\phi(t,\V \xi)=\phi(-t,-\V \xi)$.
\item $\Phi_k(\V \xi)=\int_\R \phi(t,\V \xi) t^k \dint t$ is a homogeneous polynomial in $\V \xi \in \mathbb{S}^{d-1}$ for any $k\in \N$.
\end{enumerate}
\end{theorem}
In particular, for $k=0$, we must have that $\int_{\R} \phi(t,\V \xi) \dint t= a_{\V 0}={\rm Constant}$ for all
$\V \xi \in  \mathbb{S}^{d-1}$.



The most basic version of a ridge is $\delta(\V \xi_0^\Top \V x-t_0)$ with $r=\delta(\cdot-t_0)$, which is a Dirac ridge along $\V \xi_0$  with offset $t_0$.
Since the Fourier transform of such ridges is entirely localized along the ray $\{\bw=\omega\V \xi_0: \omega \in\R\}$, we can expect their Radon transform to vanish 
away from $\pm \V \xi_0$. 
The latter can be readily identified as follows, where the square bracket notation $[g]$ with 
$ g \in \Spc S'(\R \times \mathbb{S}^{d-1})$ reminds us that the members of
$\Spc S'_{\rm Rad}$ (resp., of $\Op K_{\rm rad}\Op R\big(\Spc S(\R^d)\big)'$) 
are equivalence classes of distributions.
\begin{proposition}[Radon transform of ridge distributions]
\label{Theo:Rad1DprofilesNew}
Let $(t_0,\V \xi_0)=\V z_0 \in \R \times \mathbb{S}^{d-1}$ and $r \in \Spc S'(\R)$.
Then,
\begin{align*}
\Op K_{\rm rad}\Op R \{ \delta(\V \xi_0^\Top\V x- t_0) \}(t,\V \xi)&=
[\delta(\cdot - t_0)\delta(\cdot-\V \xi_0)]\quad  \in \Spc S'_{\rm Rad}\\
\Op R\{ \delta(\V \xi_0^\Top\V x) \}(t,\V \xi)
&=
[q_d(t)\delta(\V \xi-\V \xi_0)]\quad  \in \Op K_{\rm rad}\Op R\big(\Spc S(\R^d)\big)'\\
\Op K_{\rm rad}\Op R\{ r(\V \xi_0^\Top\V x) \}(t,\V \xi)
&=
[ r(t)\delta(\V \xi-\V \xi_0)]\quad  \in \Spc S'_{\rm Rad}\\
\Op R \{ r(\V \xi_0^\Top\V x) \}(t,\V \xi)
&= [(q_d \ast r)(t)\delta(\V \xi-\V \xi_0)]\quad  \in \Op K_{\rm rad}\Op R\big(\Spc S(\R^d)\big)'\end{align*}
where $q_d(t)=2(2\pi)^{d-1}\Fourier^{-1}\{ 1 / |\omega|^{d-1}\}(t)$ is the 1D impulse response of the Radon-domain inverse filtering operator $\Op K_{\rm rad}^{-1}$.
\end{proposition}

\begin{proof}
The fact that $
\delta_{\V z_0}=\delta(t-t_0)\delta(\V \xi-\V \xi_0)$ is a formal filtered projection of $\delta(\V \xi_0^\Top\V x- t_0)$ has already been mentioned in the text---it is a direct consequence of Definition \eqref{Eq:Radon2}.

To derive the third identity, we observe that, for all $\varphi \in \Spc S(\R^d)$, one has that
\begin{align*}
\langle r(\cdot)\delta(\cdot-\V \xi_0),\Op R\{\varphi\}\rangle_{\rm Rad}&=\int_{\R} r(t)\Op R\{\varphi\}(t,\V \xi_0)\dint t=\langle r, \Op R\{\varphi\}(\cdot,\V \xi_0) \rangle\\
&=\int_{\R^d} r(\V \xi_0^\Top \V x) \varphi(\V x) \dint \V x=\langle r(\V \xi_0^\Top \cdot),  \Op R^\ast \Op K_{\rm rad}\{\Op R\varphi\}\rangle,
\end{align*}
where we have made use of \eqref{Eq:Ridges}.
By setting $\phi=\Op R\{ \varphi\}\in \Spc S_{\rm rad}$, we then refer to \eqref{Eq:KRDist} to deduce that
$r(t)\delta(\V \xi-\V \xi_0)$ is a formal filtered projection of $r(\V \xi_0^\Top\V x)$.
This then also yields the second and fourth identities by substituting $r$ with
$q_d(\cdot-t_0)$ and $q_d \ast r$, respectively.
The additional element there is $r(\V \xi_0^\Top\V x)=\Op K \Op K^{-1}\{r(\V \xi_0^\Top\V x)\}(\V x)=
\Op K\{\tilde r(\V \xi_0^\Top\V x)\}(\V x)$ with $\tilde r(t)=(q_d \ast r)(t)$, which is readily verified in the Fourier domain.

\end{proof}


An equivalent form of the first identity in Proposition \ref{Theo:Rad1DprofilesNew} is
\begin{align}
\delta(\V \xi_0^\Top\V x- t_0) &=\Op R^\ast\{\delta_{(t_0,\V \xi_0)}\}(\V x),
\label{Eq:RastDelta}\end{align}
which results from 
$\Op R^\ast\Op K_{\rm rad}\Op R=\Identity$ on $\Spc S'(\R^d)$. Likewise, when the back-projection in \eqref{Eq:RastDelta}  is followed by an isotropic operator $\Lop$ whose radial frequency response is $\widehat L_{\rm rad}: \R \to \R$, we use the intertwining property  to show that
\begin{align}
\label{Eq:ridgebackprop}
\Lop\Op R^\ast\{\delta_{(t_0,\V \xi_0)}\}(\V x)=\Op R^\ast\{\Lop_{\rm rad}\delta_{(t_0,\V \xi_0)}\}(\V x)=
r(\V \xi_0^\Top\V x- t_0)
\end{align}
where $r(t)=\Fourier^{-1}\{ \widehat L_{\rm rad}\}(t)$.

\subsection{Connection with Prior Works}
Most authors who use the Radon transform in connection with neural networks do not distinguish ``formal'' from ``range-compatible'' Radon transforms of distributions \citep{Candes1999_ridgelets,Sonoda2017,Ongie2020b,Parhi2021b}. They bypass the difficulty by focusing their analysis on some appropriate subspace of
$\Spc S'_{\rm Rad}$ over which the backpropagation operator $\Op R^\ast$ is known to be invertible, for instance $\Spc M_{\rm even}(\R \times \mathbb{S}^{d-1})=\Spc M(\mathbb{P}^d)$ and/or $\Spc S'_{\rm Liz, even}$.
To specify their norm for Radon-domain measures, \citet{Ongie2020b} consider  the subspace of test functions
$\Spc S_{\rm even}=\{\phi \in \Spc S(\R \times \mathbb{S}^d): \phi(t,\V \xi)=\phi(-t,-\V \xi)\}$ for which the validity of the inversion formula $\Op K_{\rm rad}\Op R\Op R^\ast=\Identity$
has been established by 
\citet{Solmon1987}. The caveat is that the functions $f \in \Op R^\ast(\Spc S_{\rm even})$ that are in the range of $\Op R^\ast$ can decay as badly as $O(1/\|\V x\|)$ \citep[Corollary 3.1.1, p. 73]{Ramm2020}, meaning that their ``classical'' Radon transform specified by \eqref{Eq:Radon2}  can be ill-defined.
\citet{Parhi2021b} follow a different path and identify $\Spc M_{\rm even}(\R \times \mathbb{S}^{d-1})$ as a subspace of the space of even Lizorkin distributions $\Spc S'_{\rm Liz, even}$, which is the topological dual of $\Spc S_{\rm Liz, even}=\{\phi \in \Spc S_{\rm even}: \int_{\R} \phi(t,\V \xi)t^k=0, \forall k\in \N, \V \xi  \in \mathbb{S}^{d-1}\}$.
Implicit in the calculation of Example 1 in \citep{Ongie2020b}  is the property that
\begin{align}
\label{Eq:RadonDirac}
\Op K_{\rm rad}\Op R \{ \delta(\V \xi_0^\Top\V x- t_0) \}(t,\V \xi)
=\frac{1}{2} \Big(\delta(t-t_0)\delta(\V \xi-\V \xi_0)  + \delta(t+t_0)\delta(\V \xi+\V \xi_0)\Big),
\end{align}
which needs to be related to the ``abstract'' description of $e_{\V z_0}$ 
in Theorem \ref{Theo:Extreme0}. 
%
As it turns out, 
the two forms are equivalent. The abstract version, combined with the properties in Theorem \ref{Theo:Extreme0}, 
conveys the same information as \eqref{Eq:RadonDirac}. In complement is the concrete representation of $e_{(t_0,\V \xi_0)}(t,\V \xi)$ given by \eqref{Eq:ApproxDirac}, which gives us a sense of how and why the Radon-domain mass concentrates around the points
$\pm\V z_0$ as the Gaussian ``blob'' $g_\epsilon$ gets thinner along the primary axis and elongates in the perpendicular directions.

Even though the test functions used in the listed works are different from ours
with $\Spc S_{\rm Liz}\subset\Spc S_{\rm Rad}\subset \Spc S_{\rm even}$, the approaches are reconciled by invoking the property that
\begin{align}
C_{0, {\rm even}}=\overline{(\Spc S_{\rm even},\|\cdot\|_{L_\infty})}=\overline{(\Spc S_{\rm Liz},\|\cdot\|_{L_\infty})}=\overline{(\Spc S_{\rm Rad},\|\cdot\|_{L_\infty})}=C_{0,{\rm Rad}}
\label{Eq:ClosureC0}
\end{align}
where the domain of all underlying spaces is $(\R \times \mathbb{S}^{d-1})$  \citep[Lemma 1]{Neumayer2022}. While Proposition \ref{Theo:Rad1DprofilesNew} gives the filtered Radon transform of ridges with the greatest possible level of generality, the caveat is that $\Op K_{\rm rad}\Op R \{r_{\V \xi_0} \}\in \Spc S'_{\rm Rad}$ is an abstract equivalence class.
As complement, we are providing the ``concrete'' version of the main result for the case where the profile is a measure.

\begin{corollary}[Filtered Radon Transform of Ridge Measures]
\label{Theo:Rad1DprofilesContinuous}
Let $r_{\V \xi_0}=r(\V \xi_0^\Top\V x)$ be the ridge with profile $r\in \Spc S'(\R)$ and direction $\V \xi_0 \in \mathbb{S}^{d-1}$. If $r \in \Spc M(\R)$, then the equality
\begin{align}
\label{Eq:RidgeContinuous}
\Op K_{\rm rad}\Op R \{r_{\V \xi_0} \}(t,\V \xi)&=\frac{1}{2} \big(r(t)\delta(\V \xi-\V \xi_0) + r(-t)\delta(\V \xi+\V \xi_0)\big)
\end{align}
holds in $\Spc M_{\rm Rad}=\Spc M_{\rm even}(\R \times \mathbb{S}^{d-1})$. 
\end{corollary}
Indeed, we know that $r(t)\delta(\V \xi-\V \xi_0)$ is a formal filtered Radon transform of $r_{\V \xi_0}$ and that it is included in $\Spc M(\R \times \mathbb{S}^{d-1})$ if $r \in \Spc M(\R)$.
Moreover, Theorem \ref{Theo:Complementedspaces} with $\Spc X=C_0(\R \times \mathbb{S}^{d-1})$, together with \eqref{Eq:ClosureC0}, implies that $\Spc M_{\rm Rad}=(C_{0,{\rm even}})'=\Spc M_{\rm even}$, whose Banach complement in $\Spc M(\R \times \mathbb{S}^{d-1})$ is  $\Spc M_{\rm odd}$. This ensures the validity of the second part of Theorem \ref{Theo:Complementedspaces} with $\Op P^\ast_{\rm Rad}=\Op P_{\rm even}$, where $\Op P_{\rm even}\{g\}(t,\V \xi)=\tfrac{1}{2} \big(g(t,\V \xi)+g(-t,-\V \xi)\big)$.
Accordingly, we can identify $\Op P_{\rm even}\{r(t)\delta(\V \xi-\V \xi_0)\}$ 
as the unique representer in $\Spc M_{\rm Rad}$ of $\Op K_{\rm rad}\Op R \{r_{\V \xi_0} \}=[r(t)\delta(\V \xi-\V \xi_0)]\in \Spc S_{\rm Rad}'$.


The argument also suggests that \eqref{Eq:RidgeContinuous} is likely to be extendable to broader families of distributions. The condition for its validity is that $r(t)\delta(\V \xi-\V \xi_0)$ 
be included in a space $\Spc X'=\big(\Spc X_{\rm Rad} \oplus \Spc X^{\rm c}_{\rm Rad}\big)'$ 
with $\Spc X_{\rm Rad}=\overline{(\Spc S_{\rm Rad},\|\cdot\|_{\Spc X} )}=\overline{(\Spc S_{\rm even},\|\cdot\|_{\Spc X} )}$, so that the underlying projector can be readily identified as $\Op P^\ast_{\rm Rad}
=\Op P_{\rm even}: \Spc X' \to  \Spc X'_{\rm Rad}$.

\section{Variational Optimality of ReLU Networks}
\label{Sec:RadonRegul}

As application of the proposed formalism, we shall now link ReLU neural networks with functional optimization, revisiting the  energy-minimization property uncovered in \citep{Ongie2020b} as well as the general variational-learning problem investigated in \citep{Parhi2021b}. 
\subsection{Learning with Radon-Domain Regularization}
In order to state the relevant optimization problem, we consider the regularization operator 
\begin{align}
\Delta_{\rm R}=\Op K_{\rm rad}\Op R  \Delta: \Spc M_{\Op \Delta_{\rm R}}(\R^d) \to \Spc M(\R \times \mathbb{S}^{d-1})
\end{align}
that was first proposed by 
\citet[Lemma 2 p. 6]{Ongie2020b}, where $\Delta$ is the $d$-dimensional Laplace operator. The $\Spc M$-norm (a.k.a.\ total variation)
of this ``Radonized'' Laplacian is then used as regularization. The corresponding native space $\Spc M_{\Op \Delta_{\rm R}}(\R^d)$ is the largest Banach space over which this seminorm is well-defined under the constraint that the null space of $\Delta_{\rm R}$ be limited to polynomials of degree $1$.

\begin{theorem} [Optimality of Shallow ReLU Networks
]
\label{Theo:ReLUSplines}
Let $E: \R \times \R \to \R$ be a strictly convex loss function,
 $(\V x_m,y_m)\in \R^d \times \R$ with $m=1,\dots,M$ a given set of distinct data points, and $\lambda>0$ some fixed regularization parameter.
Then, for $M>d+1$, the solution set \begin{align}
\label{Eq:ReLUEnergy}
S=\left\{\arg \min_{f \in \Spc M_{\Delta_{\rm R}}(\R^d) } \left( \sum_{m=1}^M E(y_m ,f(\V x_m)) +  
\lambda \|\Delta_{\rm R} f\|_{\Spc M}
\right)\right\},
\end{align}
of the
functional optimization problem
is nonempty and weak* compact. It is the weak* closure of the convex hull of its extreme points, which all take the form
\begin{align}
f_{\rm ridge}(\V x)= b+ \V b^\top \V x + \sum_{k=1}^{K_0} a_k 
{\rm ReLU}(\V \xi_k^\Top\V x- \tau_k)
\label{Eq:ExtremesplineReLU}
\end{align}
with $(b, \V b)\in \R\times \R^d$, and a number $K_0<M$ of adaptive ridges with weight, direction, and offset parameters $(a_k,\V \xi_k,\tau_k) \in \R \times \mathbb{S}^{d-1}\times \R$. 
The corresponding regularization cost, which is common to all solutions, is $\|\Delta_{\rm R} f_{\rm ridge}\|_{\Spc M}=\sum_{k=1}^{K_0}|a_k|$.
\end{theorem}

The key here is that the search space $\Spc M_{\Delta_{\rm R}}(\R^d)$ is isometrically isomorphic to $\Spc M_{\rm Rad} \times \Spc P_1$, where 
$\Spc M_{\rm Rad}=\big(C_{0, \rm Rad}\big)'$
(see Section \ref{Sec:ComplementedSpaces} with $\|\cdot\|_{\Spc X}=\|\cdot\|_{L_\infty}$
and  $\Spc X_{\rm Rad}=C_{0, {\rm Rad}}$) 
and 
$\Spc P_1$ is the space of polynomials of degree $1$ specified by
\eqref{Eq:PolNullspace} with $n_0=1$. An equivalent representation of the latter space of affine functions on $\R^d$ is $\Spc P_1=\{b + \V b^\Top \V x\}$ with $b=b_\V 0$ and $\V b=(b_{\V e_i})_{i=1}^d$, which matches the leading term in \eqref{Eq:ExtremesplineReLU}.

The crucial element for this construction is the pseudoinverse operator
\begin{align}
\Delta_{\rm R}^{\dagger}\eqdef(\Identity -\Proj_{\Spc P_1}) \Delta^{-1}\Op R^\ast: \Spc M_{\rm Rad} \to 
\Spc M_{\Delta_{\rm R}}(\R^d),
\end{align}
where $\Proj_{\Spc P_1}: \Spc S'(\R^d) \to \Spc P_1$ is the projector defined by \eqref{Eq:ProjP}
with $n_0=1$.
The native space is then given by
\begin{align}
\Spc M_{\Op \Delta_{\rm R}}(\R^d)&=\Delta_{\rm R}^{\dagger}\big(\Spc M_{\rm Rad}
\big) \oplus \Spc P_1 \nonumber\\
&=\{\Delta_{\rm R}^{\dagger} \{w\} + p_0: (w,p_0) \in \Spc M_{\rm Rad}
 \times \Spc P_1\} \label{Eq:NativeLaplace}
\end{align}
equipped with the composite norm induced by $\Spc M_{\rm Rad} \times \Spc P_1$. 
Equivalently, given the generic form of $f$ in \eqref{Eq:NativeLaplace}, it is possible to retrieve the components $(w,p_0)$ with the help of suitable linear maps. Specifically, letting $f=\Delta_{\rm R}^{\dagger} \{w\} + p_0$, one verifies that
 \begin{align*}
\Delta_{\rm R} \{f\}&=\Delta_{\rm R}\{ \Delta_{\rm R}^{\dagger} w\} + \Delta_{\rm R}\{ p_0\} \\
&=\Op K_{\rm rad}\Op R  \Delta(\Identity -\Proj_{\Spc P_1}) \Delta^{-1}\Op R^\ast \{w\} + 0\\
&=\Op K_{\rm rad}\Op R  \Delta\Delta^{-1}\Op R^\ast \{w\} - \Op K_{\rm rad}\Op R  \underbrace{\Delta \{\Proj_{\Spc P_1} \Delta^{-1}\Op R^\ast w\}}_{0}\\
&=\Op K_{\rm rad}\Op R \Op R^\ast \{w\}=\Op P^\ast_{\rm Rad} \{w\}=w \in \Spc M_{\rm Rad}
\end{align*}
and
\begin{align*}
\Proj_{\Spc P_1}\{f\}&=\Proj_{\Spc P_1} \{\Delta_{\rm R}^{\dagger} w\} + \Proj_{\Spc P_1} \{p_0\}\\
&=\Proj_{\Spc P_1}(\Identity -\Proj_{\Spc P_1}) \Delta^{-1}\Op R^\ast \{w\}  + p_0 
=0 + p_0 \in \Spc P_1,\end{align*}
where the annihilation of the $w$-component follows from the
idempotence of the projector. 

\vspace*{1ex}
\begin{proof}[Proof of Theorem \ref{Theo:ReLUSplines}]
Theorem  \ref{Theo:Complementedspaces} with $\Spc X=C_0(\R \times  \mathbb{S}^{d-1})$
ensures that the back-projection operator $\Op R^\ast$ is invertible on $\Spc X'_{\rm Rad}=\Spc M_{\rm Rad}=(C_{0,\rm Rad})'$.
This is the fundamental ingredient that makes
$\Spc M_{\Delta_{\rm R}}(\R^d)$ isometrically isomorphic to $\Spc M_{\rm Rad} \times \Spc P_1$ via the reversible mapping $w=\Delta_{\rm R}\{f\}\in  \Spc M_{\rm Rad}$ and $p_0=\Proj_{\Spc P_1}\{f\}\in \Spc P_1$.
This equivalent representation of $f$ enables us to derive the result as a corollary of the third case of the abstract representer theorem for direct sums in 
\citep[Theorem 3]{Unser2022}. This representer theorem gives the generic form of the extreme points of the solution set $S$ as $f_{\rm extreme}=p_0 + \sum_{k=1}^{K_0} a_k e_k$, where $p_0 \in \Spc P_1$ is a null-space component and where the $e_k$ are extreme points of the unit ball $B_{\Spc U'}$ of the primary-component space $\Spc U'=\Delta_{\rm R}^{\dagger}\big(\Spc M_{\rm Rad}\big)$.
Based on the form of the extreme points of $\Spc M_{\rm Rad}$ given by Theorem \ref{Theo:Extreme0} and the property that ${\rm Ext}B_{\Spc U'}=\Delta_{\rm R}^{\dagger}\big({\rm Ext}B_{\Spc M_{\rm Rad}}\big)$ (since $\Delta_{\rm R}^{\dagger}$ is an isometry), we deduce that any $e_k \in {\rm Ext}B_{\Spc U'}$ can be written as
\begin{align*}
e_k&=\pm \Delta_{\rm R}^{\dagger}\{e_{(t_k,\V \xi_k)}\}=(\Identity -\Proj_{\Spc P_1}) \Delta^{-1}\Op R^\ast\{e_{(t_k,\V \xi_k)}\}\\
&= \pm (\Identity -\Proj_{\Spc P_1})\{\tfrac{1}{2}|\V \xi_k^\Top\cdot-t_k|\} \\
&=\pm \tfrac{1}{2}|\V \xi_k^\Top\cdot -t_k| \pm p_k\quad\mbox{ with } \quad p_k\in \Spc P_1,
\end{align*}
where we have used \eqref{Eq:ridgebackprop} with $\Lop=\Delta^{-1}$ and $\widehat L_{\rm rad}(\omega)=1/\omega^2$ to evaluate the back-projection.
Since $\tfrac{1}{2}|t|=(t)_+-\tfrac{1}{2}t$, we can also write $e_k$ as
\begin{align*}
e_k&=\pm (\V \xi_k^\top \cdot - t_k)_+  \pm \tilde p_k,
\end{align*}
where $\tilde p_k=p_k-\tfrac{1}{2}  \left(\V \xi_k^\top \cdot + t_k\right) \in \Spc P_1$. We then obtain \eqref{Eq:ExtremesplineReLU} by forming
the linear combination of extreme points and by grouping all polynomial components
in a single term $b + \V b^\Top \V x$. Since $\Delta_{\rm R}\{f_{\rm ridge}\}=\sum_{k=1}^{K_0} a_k e_{(t_k,\V \xi_k)}$, we also deduce that
$\|\Delta_{\rm R}\{f_{\rm ridge}\}\|_{\Spc M}=\sum_{k=1}^{K_0} |a_k|$  by invoking Theorem \ref{Theo:Extreme0}.
\end{proof}

\subsection{Discussion}

Our formulation of Theorem \ref{Theo:ReLUSplines} 
owes a lot to the pioneering works of 
\citet{Ongie2020b} and 
\citet{Parhi2021b} (PN). Our is merely a refinement of the results
published by these authors together with a clarification of the underlying mathematics. 
The interesting outcome is that the solution of the functional-optimization problem in \eqref{Eq:ReLUEnergy} can be implemented by a 2-layer ReLU network. 

In their work which pulls together ideas from  \citet{Unser2017} and \citet{Ongie2020b}, PN restrict the domain of the test functions to the so-called Lizorkin functions 
\begin{align}
\Spc S_{\rm Liz}(\R^d)=\{\varphi \in \Spc S(\R^d): \int_{\R^d} \V x^{\V m} \varphi(\V x) \dint \V x=0, \forall \V m\in \N^d\},
\end{align} which are orthogonal to the polynomials. This choice is motivated by the property that 
the Radon transform is a homeomorphism
$\Op R: \Spc S_{\rm Liz}(\R^d) \to \Spc S_{\rm Liz, even}$, where
$\Spc S_{\rm Liz, even}$ denotes the even part of the Radon-domain Lizorkin space $\Spc S_{\rm Liz}(\R \times    \mathbb{S}^{d-1})$ with $\Op R^\ast \Op R \Op K_{\rm rad}=\Identity$ on $\Spc S_{\rm Liz, even}$, as well as on
$\Spc S'_{\rm Liz, even}$, by duality \citep{Helgason2011,Kostadinova2014}. 

While the adoption of this formalism  leads to a well-defined functional-optimization problem, PN's derivation/interpretation of Lemmas 17, 18, and 21  is flawed because they implicitly assume that there is a systematic, one-to-one association between
a ``concrete'' spline ridge  $\rho_{m}(\V \xi_0^\Top \V x-t_0)$  with $\rho_m(t)=\tfrac{1}{2}{\rm sign}(t)\frac{t^{m-1}}{m!}$ and some abstract Lizorkin distribution $\rho_{m}(\V \xi_0^\Top \V x-t_0)+\Spc P \in \Spc S_{\rm Liz}'(\R^d)$, which is unlikely to be the case for the reasons exposed in the introduction. We have recent evidence that such an association can be made, but that it requires a specific polynomial correction that depends on the shift $t_0$ \citep{Neumayer2022}.
%
That said, it remains that the main results and conclusions reported by 
\citet{Parhi2021b} 
are 
qualitatively correct 
and in agreement with Theorem \ref{Theo:ReLUSplines} (up to the mentioned technicalities). Also, the mathematical arguments proposed by these authors can easily be corrected/upgraded by extending their space of test functions to $\Spc S_{\rm Rad}$ and by using our results in  Theorems \ref{Theo:Complementedspaces}
and \ref{Theo:Extreme0}. 
The same holds true for PN's higher-order generalizations (ridge splines). In fact, PN's definition of the native space of the $m$th-order ridge splines is equivalent to that of $\Spc M_{\Op \Delta_{\rm R}}(\R^d)$ for $m=2$.


Now, the one aspect where Theorem \ref{Theo:ReLUSplines} improves upon 
\citet[Theorem 1]{Ongie2020b} and \citet[Theorem 1 with $m=1$]{Parhi2021b} is that it contains the characterization of the full solution set. The theorem tells us that {\em all extreme points} of the optimization problem in \eqref{Eq:ReLUEnergy} have the same parametric form \eqref{Eq:ExtremesplineReLU}, which is a much stronger statement than the existence of one such neural-network-like solution. Ideally, one would like to identify the sparsest solution within the solution set, in other words, the one with the fewest neurons. While there is a direct algorithm that will find the sparsest solution for $d=1$ \citep{Debarre2020}, it is not known yet if this can be generalized to a larger number of dimensions.

We also like to point out that the functional-learning problem in \eqref{Eq:ReLUEnergy} is invariant to similarity transformations of the data points $\V x_m$. This is to say that any such transformation of the data characterized by a scale $s$ can be accounted for via a proper rescaling of the regularization parameter $\lambda \to s \lambda$.

\begin{proposition}
\label{Prop:Similarity}
The regularization functional in \eqref{Eq:ReLUEnergy} is translation-, scale- and rotation-invariant in the sense that
\begin{align}
\|\Delta_{\rm R}\{f(s\M U \cdot -\M b)\}\|_{\Spc M}=s \|\Delta_{\rm R}\{f\}\|_{\Spc M}
\end{align}
for any $f \in \Spc M_{\Delta_{\rm R}}(\R^d)$, and any scaling factor $s \in \R^+$, offset $\M b \in \R^d$, and rotation matrix
$\M U \in \R^{d \times d}$ with $\M U^\Top\M U=\M I$.
\end{proposition}

There has been concern 
about the well-posedness of the generative model used in \citep{Parhi2021b} and summarized by $f=\Delta_{\rm R}^{\dagger}\{w\}$ with $w \in \Spc M_{\rm Rad}=\Spc M_{\rm even}$ in the present formulation.
To make the link with \citep{Bartolucci2021}, it is instructive to describe our inverse operator $\Delta_{\rm R}^{\dagger}=(\Identity -\Proj_{\Spc P_1}) \Delta^{-1}\Op R^\ast: \Spc M_{\rm Rad} \to 
\Spc M_{\Delta_{\rm R}}(\R^d)$ explicitly in terms of the generic integral equation
\begin{align}
\label{Eq:IntOp}
\Delta_{\rm R}^{\dagger}\{w\}(\V x)=\int_{\R} \int_{\mathbb{S}^{d-1}} h(\V x; t,\V \xi) w(t, \V \xi) \dint \V \xi \dint t,
\end{align}
with a ``kernel'' $h: \R^d \times (\R \times \mathbb{S}^{d-1}) \to \R$ that is given by
\begin{align}
\label{Eq:Radkernel}
h(\V x; t,\V \xi)&= \tfrac{1}{2}|\V \xi^\Top\V x-t| -\sum_{|\V k|\le1} \frac{\V x^\V k}{\V k!}q_\V k(t,\V \xi),
\end{align}
where $q_\V k(t,\V \xi)=\langle  \tfrac{1}{2}|\V \xi^\Top\cdot -t|,m_\V k^\ast \rangle$ and where the dual basis $(m^\ast_\V k)_{|\V k|\le 1}$ is specified by 
\eqref{Eq:Dualbasis}. 
This form is compatible with \citep[Lemma 21]{Parhi2021b} and has been interpreted as an infinite-width neural network. By evaluating the underlying duality products explicitly and by regrouping the first-order correction terms, we obtain the simplified formula
\begin{align}
\label{Eq:Radkernel2}
h(\V x; t,\V \xi)&=\tfrac{1}{2}|t-\V \xi^\Top\V x| - (\kappa_{\rm rad}\ast \tfrac{1}{2}|\cdot|)(t)  + (\V \xi^\Top\V x)(\kappa_{\rm rad}\ast \tfrac{1}{2}{\rm sign})(t),
\end{align}
where the radial convolution with $\kappa_{\rm rad}$ has a mollifying effect on the two correction terms. 
If we fix $t,\V \xi$ in \eqref{Eq:Radkernel2}, then $\V x \mapsto h(\V x; t, \V \xi)$ grows like $O(|\V \xi^\Top \V x|)$ with the correction contributing a first-order polynomial. The effect of the correction is more remarkable along the radial variable $t$, in that it neutralizes the growth of the leading term $\tfrac{1}{2}|t-\V \xi^\Top\V x|=O(|t|)$: This is because the profile $t \mapsto h(\V x; t, \V \xi)$ is bounded with a maximum proportional to $|\V \xi^\Top  \V x|$ around the origin, continuous (due to the convolution of ${\rm sign}$ with $\kappa_{\rm rad}$), and vanishing towards $0$ as $t\to\pm \infty$.
This can be translated into the following 
properties, which are the mathematical bases for the present construction.
\begin{enumerate}
\item For any $\V x_0 \in \R^d$, $h(\V x_0; \cdot,\cdot)=\Delta_{\rm R}^{\dagger\ast}\{\delta(\cdot-\V x_0)\} \in C_0(\R \times \mathbb{S}^{d-1})$, which is equivalent 
to the weak* continuity of the sampling functionals $\delta(\cdot-\V x_0): f \mapsto f(\V x_0)$; that is,
$\delta(\cdot-\V x_0)\in C_{0,\Delta_{\rm R}}(\R^d)=\Delta_{\rm R}^{\ast}\big(C_{0,\rm Rad}
\big)\oplus \Spc P_1' $ with $\Spc P'_1={\rm span}\{m^\ast_\V k\}_{|\V k|\le 1}$, where $C_{0,\Delta_{\rm R}}$ is the predual of $\Spc M_{\Delta_{\rm R}}$. 
\item Stability: the integral operator \eqref{Eq:IntOp} satisfies the bound
\begin{align}
\big|\Delta_{\rm R}^{\dagger}\{w\}(\V x)\big|\le  \|h(\V x; \cdot,\cdot)\|_{L_\infty} \; \|w\|_{\Spc M} \le C (1 + \|\V x\|) \; \|w\|_{\Spc M(\R \times \mathbb{S}^{d-1})},
\end{align}
which ensures that $\Delta_{\rm R}^{\dagger}\{w\}(\V x)$ is well-defined (and continuous) for any $\V x \inR^d$ and 
$w \in \Spc M_{\rm even}(\R \times \mathbb{S}^{d-1})$. 
\end{enumerate}
Note that one can also take $w \in \Spc M(\R \times \mathbb{S}^{d-1})$ in \eqref{Eq:IntOp} since the null space of $\Delta_{\rm R}^{\dagger}: \Spc M(\R \times \mathbb{S}^{d-1}) \to \Spc M_{\Delta_{\rm R}}(\R^d)$ \big(and of $\Op R^\ast: \Spc M(\R \times \mathbb{S}^{d-1}) \to \Spc S'(\R^d)$\big) is precisely the complementary space $\Spc M_{\rm Rad^\perp}=\Spc M_{\rm odd}(\R \times \mathbb{S}^{d-1})$ (see Theorem \ref{Theo:Complementedspaces}, Item 6).

An alternative to \eqref{Eq:Radkernel2} is to consider a ``tempered'' kernel of the form
$h_{\beta}(\V x; t,\V \xi)= \frac{1}{2}|\V \xi^\Top\V x-t|\, \beta(t)$, where $\beta(t)>0$ is a weighting function (e.g., $\beta(t)= \frac{1}{1 + |t|^{2+\epsilon}})$ that compensates the linear growth of the first factor\footnote{The first factor can also be  replaced by $(\V \xi^\Top\V x-t)_+$ modulo some adjustments in $\beta$, as shown by the authors. 
}\citep{Bartolucci2021}. In effect, this mechanism, whose stability is intrinsically guaranteed, reduces the size of the native space. Interestingly, the corresponding optimization problem admits the same form of solution---a neural network with one hidden ReLU layer  \citep{Bartolucci2021}---with the caveat that the underlying regularization is no longer translation-invariant. In this modified scenario, the optimal cost is $\|\Delta_{\rm R} f_{\rm ridge}\|_{\Spc M_{1/\beta}}=\sum_{n=1}^{K_0}|a_k| \frac{1}{|\beta(\tau_k)|}$, which tends to favor smaller biases $\tau_k$. This also means that one then looses the invariance to similarity transformations of the data (Proposition \ref{Prop:Similarity}).

\appendix
\section{Direct-Sum Topologies}
\label{Sec:DirectSums}
There are two standard ways to define direct sums: explicitly, via the use of projectors; or abstractly, via the use of quotient spaces. The two methods are equivalent whenever one can explicitly identify the underlying quotient space as a (concrete) complemented subspace of the original space.
\subsection{Projectors}
Let $\Spc X$ be a topological vector space.
A continuous linear operator $\Op P : \Spc X \to \Spc X$ with the property that $\Op P =\Op P \circ \Op P=\Op P^2$ (idempotence) on $\Spc X$ is called a projection operator \citep[p 480]{Dunford1988}. In particular, when $\Spc X$ is a Banach space or a Fréchet space, the range $\Spc U=\Op P(\Spc X)$ of $\Op P$ is necessarily a closed subspace of $\Spc X$.
In that case, $\Op P=\Proj_{\Spc U}$ is a projector from $\Spc X$ onto $\Spc U$, and $\Spc X=\Spc U\oplus \Spc V$ where $\Spc V={\rm ker}\Op P$ is the null space of $\Op P$ or, equivalently, the range of the complementary projector $\Proj_\Spc V=(\Identity-\Op P)$.

More generally, when $\Spc X$ is a topological vector space, the space $\Op P(\Spc X)$ equipped with the topology induced by $\Spc X$ is a topological space as well, with the same properties as the original space $\Spc X$ (e.g., completeness). Likewise, if $(\Spc X, \Spc X')$ is a dual pair of topological spaces, then so is $\Big(\Op P(\Spc X), \Op P^\ast(\Spc X')\big)$,
where $\Op P^\ast: \Spc X' \to \Spc X'$ is the dual projection operator.

\subsection{Direct Sums}

{1) Direct-sum decomposition of a vector space $\Spc X$}: Let  $\Spc U$ and $\Spc V$ be two (complementary) closed subspaces of $\Spc X$. 
The notation $\Spc X= \Spc U \oplus \Spc V$  indicates that every element $x\in \Spc X$ has a unique decomposition as $x=u+v$
with $(u,v) \in \Spc U \times \Spc V$.
The underlying projection operators are
\begin{align*}
x=u+v &\mapsto \Proj_{\Spc U}\{x\}=u\\
x&\mapsto \Proj_{\Spc V}\{x\}=(\Identity- \Proj_{\Spc U})\{x\}=v.
\end{align*}
To summarize, given a (closed) subspace $\Spc U$ of a normed vector space $\Spc X$, the search of a complement $\Spc V$ for $\Spc U$ in $\Spc X$ is equivalent to the search of a (continuous) projection operator $\Op P$ on $\Spc X$ (with $\Op P^2=\Op P$) whose range is $\Spc U$. Then, $\Spc V=\Proj_{\Spc V}(\Spc X)$ with $\Proj_{\Spc V}=(\Identity-\Op P)$.
\\[2ex]
2) Annihilator: Let $\Spc U$ be a closed subset of $\Spc X$. One then defines $\Spc U^\perp$ as the annihilator of $\Spc U$ in $\Spc X'$, which is the subset
$$
\Spc U^\perp=\{f \in \Spc X': \langle f, u\rangle=0 \mbox{ for all } u \in \Spc U\} \subseteq \Spc X'.
$$
3) Dual space: 
The dual of the direct sum $\Spc X= \Spc U \oplus \Spc V$ is $\Spc X'= \Spc U' \oplus \Spc V'$, where $\Spc U'=\Op P^\ast(\Spc X')=\Spc V^\perp$ and
$\Spc U'=\Spc U^\perp$.
\\[2ex]
4) Quotient space: 
Under the assumption that $\Spc V$ is a closed subset of $\Spc X$, one defines the quotient space $\Spc X/\Spc V$ whose
 elements are equivalence classes (or cosets) denoted by
$[x]=x + \Spc V$.
The corresponding quotient map $q: \Spc X \to \Spc X/\Spc V$
is linear and its kernel (null space) is $\Spc V$. 
When $\Spc X$ is a Banach space, the quotient norm is
$$
\|[x]\|_{\Spc X/\Spc V}= \inf_{v\in \Spc V}\|x+v\|_{\Spc X},
$$
which measures the distance from $x$ to $\Spc V$. The quotient space $\Spc X/\Spc V$ equipped with the quotient norm is a Banach space as well. Moreover, there is a natural isomorphism between $(\Spc X/\Spc V)'$ (the dual of the quotient of $\Spc X$ by $\Spc V$) and $\Spc V^\perp$ (the annihilator of $\Spc V$ in $\Spc X'$), so that $(\Spc X/\Spc V)'\embedC \Spc X'$.

Also of relevance is the property that the bounded operators on $\Spc X$ that annihilate the elements of $\Spc V$ ``factor through" $\Spc X/\Spc V$.
Let $\Op T: \Spc X \to \Spc Y$ with $\Spc V \subseteq {\rm ker}\Op T$. Then, there exists a unique linear operator $\Op T_{q}: \Spc X/\Spc V \to \Spc Y$ such that
$\Op T_{q}q(x)=\Op T(x)$ and $\|\Op T_{q}\|=\|\Op T\|$.

The kernel of any bounded operator $\Op T: \Spc X \to \Spc Y$ is a closed subspace of $\Spc X$ \citep[Proposition 4.2, p. 172]{Markin2020}. Hence,
the quotient space $\Spc X/{\rm ker}(\Op T)$ is a vector space that is isomorphic to $\Op T(\Spc X)$.

\section{Extreme Points}
\label{Sec:ExtremePoints}
\begin{definition}[Extreme Points]
Let $C$ be a convex set of a Banach space $\Spc X$. The extreme points of $C$ are the points 
$x\in C$ such that, if there exist $x_1,x_2 \in C$ and $\theta \in (0,1)$ such that $x= \theta x_1 + (1- \theta)x_2$, then it necessarily holds that $x_1 = x = x_2$. The set of these extreme points is denoted by ${\rm Ext}(C)$.
\end{definition}

We now present a classical 
result that gives the explicit form of the extreme points of the dual $\Spc X'$
of any closed subspace $\Spc X \subseteq C(\Spc Z)$, where $C(\Spc Z)$ is the space of continuous functions $z\mapsto f(z)$ on some compact Hausdorff space $\Spc Z$ equipped with the norm $\|f\|=\sup_{z\in \Spc Z}  |f(z)|$.

\begin{lemma}[{\citep[p. 441]{Dunford1988}}]
\label{Theo:ExtremeSubspace}
Let $\Spc X$ be a closed linear manifold of the Banach space $C(\Spc Z)$ of all real continuous functions on the compact Hausdorff space $\Spc Z$. For each $z \in \Spc Z$, let the evaluation functional $e_z \in \Spc X'$ be defined by
\begin{align}
\langle e_z, f\rangle=f(z),\quad  f\in \Spc X.
\end{align}
Then, every extreme point of the closed unit ball in $\Spc X'$,
\begin{align}
B_{\Spc X'}=\{x^\ast \in \Spc X': 
\|x^\ast\|_{\Spc X'}= \sup_{f \in \Spc X: |f(z)|\le 1} \langle x^\ast, f \rangle \le 1\},
\end{align} 
is of the form $\pm e_z$ with $z \in \Spc Z$. If $\Spc X=C(\Spc Z)$, then the converse is true as well; that is,
${\rm Ext} B_{\Spc M(\Spc Z)}=\{\pm e_z: z \in \Spc Z\}$ with $\Spc M(\Spc Z)=\big(C(\Spc Z)\big)'$.
\end{lemma}
Lemma \ref{Theo:ExtremeSubspace} generalizes to $C_0(\Spc Z)$, where $\Spc Z$ is a locally compact Hausdorff space, which covers the case that is of interest to us: 
$\Spc Z=\R \times \mathbb{S}^{d-1}$.

\begin{proof}
Let $E$ be the set of all points in $\Spc X'$ of the form $\pm e_z$ with $z \in \Spc Z$.
The space $\Spc X'$ is equipped with its weak$\ast$ (or $\Spc X$) topology for the Krein-Milman theorem to apply.
As $\|e_z\|_{\Spc X'}\le 1$, $E \subseteq B_{\Spc X'}$. Since $B_{\Spc X'}$ is convex,  weak*-compact and, hence, weak*-closed, the inclusion also holds for the closed convex hull, with ${\rm cch}E \subseteq B_{\Spc X'}$. 
Next, we invoke a variant of the Hahn-Banach theorem \citep[Theorem 3.5, p. 59]{Rudin1991}.
For any $x^\ast \notin {\rm cch}E$, there exists a linear functional $f \in \Spc X$ that separates $x^\ast \in \Spc X'$ from the closed convex set ${\rm cch}E$. 
This means that there are a constant $c>0$ and some $\epsilon>0$ such that
$$
\pm f(z) 
 \le c -\epsilon < c \le \langle x^\ast,f\rangle  $$ 
for all $z \in \Spc Z$. 
Hence, $\|f\|\le (c - \epsilon)$ which, when combined with $\|x^\ast\|_{\Spc X'} \|f\| \ge c$, gives that  $\|x^\ast\|_{\Spc X'}>1$. Thus, ${\rm cch}E\supseteq B_{\Spc X'}$, from which we conclude that
${\rm cch}E=B_{\Spc X'}$. Finally, since $E$ is compact, the extreme points of ${\rm cch}E$ necessarily lie in $E$ (see \citep[Milman's theorem, p. 76]{Rudin1991}).

For the converse implication, we
invoke the Riesz-representation theorem, which allows us to represent any unit-norm functional on 
$C_0(\Spc Z)$ by a real-valued measure $\mu\in \Spc M(\Spc Z)$ of total variation $1$. If the support of $\mu$ 
consists of one point, it is a signed multiple of a Dirac mass. Otherwise, ${\rm supp}\mu$ contains two distinct points $z_1\ne z_2$. Let $U,V\subset \Spc Z$
 be disjoint neighborhoods of $z_1$ and $z_2$.
By the definition of the support, $|\mu|(U)>0$ and $|\mu|(V)>0$. Define
$t=|\mu|(U)$, which is such that $0<t<1$. Now let $\lambda=t^{-1} 
\mu\vert_{U}$
and $\nu=(1-t)^{-1}\mu\vert_{U^c}$.
Then, both $\lambda$ and $\nu$ are unit-norm functionals and
$\mu=t \lambda+(1-t)\nu$, which proves that $\mu$ is not extreme.
\end{proof}

\subsection*{Acknowlegdments}
The research was partially supported by the Swiss National Science
Foundation under Grant 200020-184646.
The author is thankful to Sebastian Neumayer and Shayan Aziznejad for very helpful discussions.

%
%
%
\bibliography{/Users/unser/MyDrive/Bibliography/Bibtex_files/Unser.bib}

\begin{thebibliography}{47}
\providecommand{\natexlab}[1]{#1}
\providecommand{\url}[1]{\texttt{#1}}
\expandafter\ifx\csname urlstyle\endcsname\relax
  \providecommand{\doi}[1]{doi: #1}\else
  \providecommand{\doi}{doi: \begingroup \urlstyle{rm}\Url}\fi

\bibitem[Alvarez et~al.(2012)Alvarez, Rosasco, and Lawrence]{Alvarez2012}
Mauricio~A Alvarez, Lorenzo Rosasco, and Neil~D Lawrence.
\newblock Kernels for vector-valued functions: A review.
\newblock \emph{Foundations and Trends in Machine Learning}, 4\penalty0
  (3):\penalty0 195--266, 2012.

\bibitem[Bach(2017)]{Bach2017}
Francis Bach.
\newblock Breaking the curse of dimensionality with convex neural networks.
\newblock \emph{Journal of Machine Learning Research}, 18:\penalty0 1--53,
  2017.

\bibitem[Barron(1993)]{Barron1993}
Andrew~R Barron.
\newblock Universal approximation bounds for superpositions of a sigmoidal
  function.
\newblock \emph{{IEEE} Transactions on Information Theory}, 39\penalty0
  (3):\penalty0 930--945, May 1993.
\newblock \doi{10.1109/18.256500}.

\bibitem[Bartolucci et~al.(2021)Bartolucci, Vito, Rosasco, and
  Vigogna]{Bartolucci2021}
Francesca Bartolucci, Ernesto~De Vito, Lorenzo Rosasco, and Stefano Vigogna.
\newblock Understanding neural networks with reproducing kernel {B}anach
  spaces.
\newblock \emph{arXiv:2109.09710}, 2021.

\bibitem[Bishop(2006)]{Bishop2006}
Christopher~M Bishop.
\newblock \emph{Pattern Recognition and Machine Learning}.
\newblock Springer, 2006.

\bibitem[Boman and Lindskog(2009)]{Boman2009}
Jan Boman and Filip Lindskog.
\newblock Support theorems for the {R}adon transform and {C}ram{\'{e}}r-{W}old
  theorems.
\newblock \emph{Journal of Theoretical Probability}, 22\penalty0 (3):\penalty0
  683--710, March 2009.
\newblock \doi{https://doi.org/10.1007/s10959-008-0151-0}.

\bibitem[Cand{\`e}s(1999)]{Candes1999_ridgelets}
Emmanuel~J Cand{\`e}s.
\newblock Harmonic analysis of neural networks.
\newblock \emph{Applied and Computational Harmonic Analysis}, 6\penalty0
  (2):\penalty0 197--218, 1999.

\bibitem[Cand{\`e}s and Donoho(1999)]{Candes1999}
Emmanuel~J Cand{\`e}s and David~L Donoho.
\newblock Ridgelets: {A} key to higher-dimensional intermittency?
\newblock \emph{Phil. Trans. R. Soc. Lond. A.}, pages 2495--2509, 1999.

\bibitem[Cand{\`e}s and Romberg(2007)]{Candes2007}
Emmanuel~J Cand{\`e}s and Justin Romberg.
\newblock Sparsity and incoherence in compressive sampling.
\newblock \emph{Inverse Problems}, 23\penalty0 (3):\penalty0 969--985, 2007.

\bibitem[Cybenko(1989)]{Cybenko1989}
George Cybenko.
\newblock Approximation by superpositions of a sigmoidal function.
\newblock \emph{Mathematics of Control, Signals and Systems}, 2\penalty0
  (4):\penalty0 303--314, 1989.

\bibitem[Debarre et~al.(2022)Debarre, Denoyelle, Unser, and
  Fageot]{Debarre2020}
Thomas Debarre, Quentin Denoyelle, Michael Unser, and Julien Fageot.
\newblock Sparsest piecewise-linear representation of data.
\newblock \emph{Journal of Computational and Applied Mathematics},
  406:\penalty0 in press, 2022.
\newblock \doi{https://doi.org/10.1016/j.cam.2021.114044}.

\bibitem[Donoho(2006)]{Donoho2006}
David~L Donoho.
\newblock Compressed sensing.
\newblock \emph{IEEE Transactions on Information Theory}, 52\penalty0
  (4):\penalty0 1289--1306, 2006.

\bibitem[Dunford and Schwartz(1988)]{Dunford1988}
Nelson Dunford and Jacob~T Schwartz.
\newblock \emph{Linear Operators, Part 1: General Theory}.
\newblock John Wiley and Sons, 1988.

\bibitem[Elad(2010)]{Elad2010b}
Michael Elad.
\newblock \emph{Sparse and Redundant Representations. From Theory to
  Applications in Signal and Image Processing}.
\newblock Springer, 2010.

\bibitem[Foucart and Rauhut(2013)]{Foucart2013}
Simon Foucart and Holger Rauhut.
\newblock \emph{A Mathematical Introduction to Compressive Sensing}.
\newblock Springer, 2013.

\bibitem[Gelfand and Shilov(1966)]{Gelfand1966}
Isreal~M Gelfand and Georgy Shilov.
\newblock \emph{Generalized Functions. {V}ol. 5. {I}ntegral Geometry and
  Representation Theory}.
\newblock Academic Press, New York, USA, 1966.

\bibitem[Helgason(2011)]{Helgason2011}
Sigurdur Helgason.
\newblock \emph{Integral Geometry and {R}adon Transforms}.
\newblock Springer, 2011.

\bibitem[Hertle(1983)]{Hertle1983}
Alexander Hertle.
\newblock Continuity of the {R}adon transform and its inverse on {E}uclidean
  space.
\newblock \emph{Mathematische Zeitschrift}, 184\penalty0 (2):\penalty0
  165--192, 1983.

\bibitem[Hornik et~al.(1989)Hornik, Stinchcombe, and White]{Hornik1989}
Kurt Hornik, Maxwell Stinchcombe, and Halbert White.
\newblock Multilayer feedforward networks are universal approximators.
\newblock \emph{Neural networks}, 2\penalty0 (5):\penalty0 359--366, 1989.

\bibitem[Kostadinova et~al.(2014)Kostadinova, Pilipovi{\'c}, Saneva, and
  Vindas]{Kostadinova2014}
Sanja Kostadinova, Stevan Pilipovi{\'c}, Katerina Saneva, and Jasson Vindas.
\newblock The ridgelet transform of distributions.
\newblock \emph{Integral Transforms and Special Functions}, 25\penalty0
  (5):\penalty0 344--358, 2014.

\bibitem[Logan and Shepp(1975)]{Logan1975}
Benjamin~F Logan and Larry~A Shepp.
\newblock Optimal reconstruction of a function from its projections.
\newblock \emph{Duke Mathematical Journal}, 42\penalty0 (4):\penalty0 645--659,
  1975.

\bibitem[Ludwig(1966)]{Ludwig1966}
Donald Ludwig.
\newblock The {R}adon transform on {E}uclidean space.
\newblock \emph{Communications on Pure and Applied Mathematics}, 19\penalty0
  (1):\penalty0 49--81, September 1966.
\newblock \doi{10.1002/cpa.3160190105}.

\bibitem[Madych(1990)]{Madych1990d}
Wolodymyr~R Madych.
\newblock Summability and approximate reconstruction from {R}adon transform
  data.
\newblock \emph{Contemporary Mathematics}, 113:\penalty0 189--219, 1990.

\bibitem[Markin(2020)]{Markin2020}
Marat~V Markin.
\newblock \emph{Elementary Operator Theory}.
\newblock De Gruyter, 2020.

\bibitem[Murata(1996)]{Murata1996}
Noboru Murata.
\newblock An integral representation of functions using three-layered networks
  and their approximation bounds.
\newblock \emph{Neural Networks}, 9\penalty0 (6):\penalty0 947--956, 1996.

\bibitem[Natterer(1984)]{Natterer1984}
Frank Natterer.
\newblock \emph{The Mathematics of Computed Tomography}.
\newblock John Willey \& Sons Ltd, 1984.

\bibitem[Neumayer and Unser(2022)]{Neumayer2022}
Sebastian Neumayer and Michael Unser.
\newblock Explicit representations for {B}anach subspaces of {L}izorkin
  distributions.
\newblock \emph{Preprint}, 2022.

\bibitem[Ongie et~al.(2020)Ongie, Willett, Soudry, and Srebro]{Ongie2020b}
Greg Ongie, Rebecca Willett, Daniel Soudry, and Nathan Srebro.
\newblock A function space view of bounded norm infinite width {ReLU} nets: The
  multivariate case.
\newblock \emph{International Conference on Representation Learning (ICLR)},
  2020.

\bibitem[Parhi and Nowak(2021)]{Parhi2021b}
Rahul Parhi and Robert~D Nowak.
\newblock Banach space representer theorems for neural networks and ridge
  splines.
\newblock \emph{Journal of Machine Learning Research}, 22\penalty0
  (43):\penalty0 1--40, 2021.

\bibitem[Pinkus(2015)]{Pinkus2015}
Allan Pinkus.
\newblock \emph{Ridge Functions}.
\newblock Cambridge Tracts in Mathematics. Cambridge University Press, 2015.

\bibitem[Poggio and Girosi(1990)]{Poggio1990}
Tomaso Poggio and Federico Girosi.
\newblock Regularization algorithms for learning that are equivalent to
  multilayer networks.
\newblock \emph{Science}, 247\penalty0 (4945):\penalty0 978--982, 1990.

\bibitem[Ramm(1996)]{Ramm1996}
Alexander~G Ramm.
\newblock Inversion formula and singularities of the solution for the
  backprojection operator in tomography.
\newblock \emph{Proceedings of the American Mathematical Society}, 124\penalty0
  (2):\penalty0 567--577, 1996.
\newblock \doi{10.1090/s0002-9939-96-03155-3}.

\bibitem[Ramm and Katsevich(2020)]{Ramm2020}
Alexander~G Ramm and Alexander~I Katsevich.
\newblock \emph{The {R}adon transform and local tomography}.
\newblock CRC Press, 2020.

\bibitem[Reed and Simon(1980)]{Reed1980}
Michael Reed and Barry Simon.
\newblock \emph{Methods of Modern Mathematical Physics. {V}ol. 1: {F}unctional
  Analysis}.
\newblock Academic Press, San Diego, 1980.

\bibitem[Rubin(1998)]{Rubin1998}
Boris Rubin.
\newblock The {C}alder{\'o}n reproducing formula, windowed {x}-ray transforms,
  and {R}adon transforms in {$L^p$}-spaces.
\newblock \emph{Journal of Fourier Analysis and Applications}, 4\penalty0
  (2):\penalty0 175--197, 1998.

\bibitem[Rudin(1991)]{Rudin1991}
Walter Rudin.
\newblock \emph{Functional Analysis}.
\newblock McGraw-Hill, New York, 2nd edition, 1991.
\newblock McGraw-Hill Series in Higher Mathematics.

\bibitem[Samko(1982)]{Samko1982denseness}
Stefan~G Samko.
\newblock Denseness of the {L}izorkin-type spaces {$\Phi_V$} in ${L_p(\R^n)}$.
\newblock \emph{Mathematical notes of the Academy of Sciences of the USSR},
  31\penalty0 (6):\penalty0 432--437, 1982.

\bibitem[Samko et~al.(1993)Samko, Kilbas, and Marichev]{Samko1993}
Stefan~G Samko, Anatoly~A Kilbas, and Oleg~I Marichev.
\newblock \emph{Fractional Integrals and Derivatives: Theory and Applications}.
\newblock Gordon and Breach Science Publishers, 1993.

\bibitem[Saneva and Vindas(2010)]{Saneva_2010}
Katerina Saneva and Jasson Vindas.
\newblock Wavelet expansions and asymptotic behavior of distributions.
\newblock \emph{Journal of Mathematical Analysis and Applications},
  370\penalty0 (2):\penalty0 543--554, October 2010.
\newblock \doi{https://doi.org/10.1016/j.jmaa.2010.04.041}.

\bibitem[Sch{\"o}lkopf et~al.(1997)Sch{\"o}lkopf, Sung, Burges, Girosi, Niyogi,
  Poggio, and Vapnik]{Scholkopf1997}
Bernhard Sch{\"o}lkopf, Kah-Kay Sung, Chris J~C Burges, Federico Girosi, Partha
  Niyogi, Tomaso Poggio, and Vladimir Vapnik.
\newblock Comparing support vector machines with {G}aussian kernels to radial
  basis function classifiers.
\newblock \emph{{IEEE} Transactions on Signal Processing}, 45\penalty0
  (11):\penalty0 2758--2765, November 1997.

\bibitem[Sch{\"o}lkopf et~al.(2001)Sch{\"o}lkopf, Herbrich, and
  Smola]{Scholkopf2001}
Bernhard Sch{\"o}lkopf, Ralf Herbrich, and Alex~J. Smola.
\newblock A generalized representer theorem.
\newblock In David Helmbold and Bob Williamson, editors, \emph{Computational
  Learning Theory}, pages 416--426. Springer Berlin Heidelberg, 2001.

\bibitem[Solmon(1987)]{Solmon1987}
Donald~C Solmon.
\newblock Asymptotic formulas for the dual {R}adon transform and applications.
\newblock \emph{Mathematische Zeitschrift}, 195\penalty0 (3):\penalty0
  321--343, 1987.

\bibitem[Sonoda and Murata(2017)]{Sonoda2017}
Sho Sonoda and Noboru Murata.
\newblock Neural network with unbounded activation functions is universal
  approximator.
\newblock \emph{Applied and Computational Harmonic Analysis}, 43\penalty0
  (2):\penalty0 233--268, 2017.

\bibitem[Unser(2021)]{Unser_2020}
Michael Unser.
\newblock A unifying representer theorem for inverse problems and machine
  learning.
\newblock \emph{Foundations of Computational Mathematics}, 21\penalty0
  (4):\penalty0 941--960, September 2021.
\newblock \doi{10.1007/s10208-020-09472-x}.

\bibitem[Unser and Aziznejad(2022)]{Unser2022}
Michael Unser and Shayan Aziznejad.
\newblock Convex optimization in sums of {B}anach spaces.
\newblock \emph{Applied and Computational Harmonic Analysis}, 56:\penalty0
  1--25, January 2022.
\newblock \doi{10.1016/j.acha.2021.07.002}.

\bibitem[Unser et~al.(2017)Unser, Fageot, and Ward]{Unser2017}
Michael Unser, Julien Fageot, and John~Paul Ward.
\newblock Splines are universal solutions of linear inverse problems with
  generalized-{TV} regularization.
\newblock \emph{SIAM Review}, 59\penalty0 (4):\penalty0 769--793, December
  2017.

\bibitem[Wahba(1990)]{Wahba1990}
Grace Wahba.
\newblock \emph{Spline Models for Observational Data}.
\newblock Society for Industrial and Applied Mathematics, Philadelphia, PA,
  1990.

\end{thebibliography}


\end{document}